\documentclass[a4paper,12pt,reqno]{amsart}

\usepackage{geometry}
\usepackage{xcolor}
\usepackage[utf8]{inputenc}
\usepackage{upgreek}
\usepackage{amsmath,mathtools,esint,amssymb,enumitem,mathrsfs}
\usepackage[numbers,sort&compress]{natbib}
\usepackage{xurl}
\usepackage{comment}

\usepackage[colorlinks,allcolors=black]{hyperref}

\numberwithin{equation}{section}

\newtheorem{theorem}{Theorem}[section]
\newtheorem{lemma}[theorem]{Lemma}
\newtheorem{proposition}[theorem]{Proposition}
\newtheorem{corollary}[theorem]{Corollary}

\theoremstyle{remark}
\newtheorem{remark}[theorem]{Remark}

\newcommand{\N}{\mathbb{N}}
\newcommand{\R}{\mathbb{R}}
\newcommand{\E}{\mathbb{E}}
\newcommand{\p}{\mathbb{P}}
\newcommand{\Z}{\mathbb{Z}}

\newcommand{\var}{\operatorname{Var}}
\newcommand{\cov}{\operatorname{Cov}}

\begin{document}

\title[Large fluctuations]{Large fluctuations of extended Rademacher random multiplicative functions}

\author{Haozhe Gou}
\address{School of Mathematics, Shandong University, Jinan 250100, China}
\address{D\'epartement de math\'ematiques et de statistique, Universit\'e de Montr\'eal, C.P.~6128, succ.~Centre-ville, Montr\'eal, QC H3C~3J7, Canada}
\email{haozhegou@gmail.com}

\author{Max Wenqiang Xu}
\address{Yau Mathematical Sciences Center, Tsinghua University, Beijing, 100084, China}
\address{Beijing Institute of Mathematical Sciences and Applications, Beijing, 101408, China}
\email{maxxu1729@gmail.com, maxxu@tsinghua.edu.cn}


\begin{abstract}
Let $f$ be an extended Rademacher random multiplicative function (RMF). We show
that, for every fixed deterministic function $V(x)$ tending to infinity,
almost surely there are arbitrarily large $x$ for which
\[
 \sum_{n\leq x}f(n)
 \geq \frac{\sqrt{x}(\log\log x)^{1/4}}{V(x)}.
\]
The corresponding negative fluctuation holds as well. 
In particular, this gives an affirmative answer to Erd\H{o}s Problem~\#1144.  As a byproduct, our result has a direct corollary giving new almost sure lower bounds $\log\log x$ on the number of sign changes of partial sums up to $x$
for all sufficiently large $x$.

\end{abstract}
\maketitle

\section{Introduction}

\subsection{Main result}
A Rademacher random completely multiplicative function
$f\colon\N\rightarrow\{-1,1\}$ is obtained by taking $(f(p))_p$ to be
independent Rademacher random variables and extending $f$ completely
multiplicatively. Throughout the paper, we refer to this model as the
extended Rademacher RMF.
The more familiar Rademacher RMF is supported only on squarefree integers, and it goes back to
\citet{Wintner1944}, who introduced it as a probabilistic model for the
M\"obius function. Although it is clear nowadays that the Rademacher RMF does not capture much behavior of M\"obius function, we believe that the extended Rademacher RMF is a good model to predict the typical behavior of Dirichlet quadratic characters. 
Write $f(p)=\epsilon_p$ and
$$S_f(x)=\sum_{n\leq x}f(n).$$  
In this paper, our purpose is to prove a
quantitative two-sided fluctuation result for $S_f(x)$. 

One can distinguish this model from the Rademacher
RMF $g(n)$ supported on squarefree integers, by noting that $g(n)=\mu^2(n)f(n)$, where $\mu$ is the M\"obius function, so
that $g(n)=0$ unless $n$ is squarefree.
The distinction is substantial. The variables $g(n)$, indexed by squarefree
$n$, has perfect orthogonality, whereas
$\E f(n)$ equals $1$ when $n$ is a square and equals $0$ otherwise; in
particular, $\E S_f(x)=\lfloor\sqrt x\rfloor$. The two Dirichlet series are
related by
\begin{equation}\label{eq:euler-factorisation-intro}
 \sum_{n\geq1}\frac{f(n)}{n^s}
 =\prod_p\left(1-\frac{\epsilon_p}{p^s}\right)^{-1}
 =\zeta(2s)\prod_p\left(1+\frac{\epsilon_p}{p^s}\right),
 \qquad \mathrm{Re}\,s>1.
\end{equation}
Thus the completely multiplicative  property introduces a  zeta factor with a pole at $s=1/2$. This pole is the analytic expression of the contribution from squares and can be a principal obstruction to transferring some 
results for $g$ directly to $f$.

The study of the large fluctuations of the partial sums of RMFs has been very active.
{In the Rademacher case, \citet{Wintner1944} and \citet{Erdos1985} obtained some of the earliest results in this direction, and \citet{Halasz1983} later obtained an important almost sure $\Omega$-result.}
\citet{Harper2013Gaussian} {subsequently improved Hal\'asz's result} by means of Gaussian comparison.
His later work proved that, for every deterministic function $V(x)$ tending to
infinity, almost surely there are arbitrarily large $x$ for which
\begin{equation}\label{eq:harper-fluctuation}
 \bigg|\sum_{n\leq x}g(n)\bigg|
 \geq
 \frac{\sqrt{x}(\log\log x)^{1/4}}{V(x)}.
\end{equation}
He also proved the corresponding result for the Steinhaus RMF; see
\cite{Harper2023Large}.  This was the first almost sure lower bound for
the partial sums which exceeds $\sqrt{x}$ by a factor tending to infinity.
The proof conditions on the values at the smaller primes and applies a
multivariate Gaussian approximation to the contribution from the remaining
large primes.  Its main difficulty is to show that sufficiently many of
the resulting conditional covariances are small.  Although the theorem is
stated in the absolute value form \eqref{eq:harper-fluctuation}, in the
Rademacher case the conditional symmetry of the large prime vector gives
the corresponding large fluctuations of either sign.

Harper further suggested that the exponent $1/4$ in
\eqref{eq:harper-fluctuation} should be optimal.  More precisely, the
conjectured  upper bound is that, for every
$\varepsilon>0$, almost surely
\begin{equation}\label{eq:harper-conjectural-upper}
 \sum_{n\leq x}g(n)
 \ll_{\varepsilon}
 \sqrt{x}(\log\log x)^{1/4+\varepsilon}.
\end{equation}
One heuristic for this exponent comes from the low moments of the partial
sums.  The results of \citet{Harper2020Low} show that their typical size is
smaller than $\sqrt{x}$ by a factor of order
$(\log\log x)^{1/4}$.  Combining this with the usual
$\sqrt{\log\log x}$ correction suggested by the law of the iterated
logarithm leads precisely to the scale
$\sqrt{x}(\log\log x)^{1/4}$.  The relevant low and high moment estimates
for the underlying random Euler products are developed in
\cite{Harper2020Low,Harper2019High}.

We give a quick summary of the substantial progress on the upper bounds: see, for example,
\cite{Halasz1983,Basquin2012,LauTenenbaumWu2013,Mastrostefano2022,Caich2023}
and the references therein.
Very recently, \citet{DurkanPearceCrump2026} obtained the matching upper
bound
\begin{equation}\label{eq:durkan-pearce-crump-upper}
 \sum_{n\leq x}g(n)
 \ll_{\varepsilon,g}
 \sqrt{x}(\log\log x)^{1/4+\varepsilon}
 \qquad\text{almost surely}.
\end{equation}
They proved the same result for the Steinhaus RMF. Together with Harper's
lower bound, this determines the optimal iterated logarithmic exponent in
both classical RMFs and settles Harper's conjecture.
In an independent work, \citet{Verreault2026} obtained the sharper estimate
\begin{equation}\label{eq:verreault-upper}
 \sum_{n\leq x}g(n)
 \ll_{\varepsilon,g}
 \sqrt{x}(\log_2 x)^{1/4}
 \log_3 x\,\log_4 x\,(\log_5 x)^{1/2+\varepsilon}
 \qquad\text{almost surely},
\end{equation}
for both the Rademacher and Steinhaus RMFs, where $\log_k$ denotes the
$k$-fold iterated logarithm.  We also note that in the weighted case, 
\citet{Atherfold2025} proved that almost surely
$$\sum_{n\leq x}\frac{g(n)}{\sqrt n}\ll(\log\log x)^{3/4+\varepsilon},$$ and obtained
the sharper exponent $1/4+\varepsilon$ under the restriction
$P^+(n)>\sqrt x$. 
He conjectured that this exponent is optimal.  He also showed that {almost surely}
$$S_f(x)\ll\sqrt{x}(\log x)^{1+\varepsilon}$$ for extended
Rademacher RMFs. See also the work by Hardy on the Steinhaus case \cite{Hardy2024Weighted}.

It is in general the case that the three standard RMFs may exhibit
 different levels of difficulty and the Steinhaus case is usually the least technical among the
three. So far, we in general know less in the Rademacher case, while the extended
Rademacher RMF is arguably even more challenging.
This is indeed the case in the fluctuation problem considered here. As we described before, large fluctuations in both the Rademacher and Steinhaus RMFs
have been known since the work of \citet{Harper2023Large}, whereas no corresponding result
was previously available for the extended Rademacher
RMF. In fact, even the weaker original question of Erd\H{o}s remained
open\footnote{After completing the paper, 
we learned through the Erd\H{o}s Problems website that, on September 6, 2026,
Sigurd H{\o}ystad independently announced a solution to the original
question, obtained with the assistance of AI tools, and claimed a complete
formalization of the proof in Lean~4; see \cite{Hoystad2026}.
Their result addresses only the original Erd\H{o}s question, i.e. a qualitative result, whereas the
present work proves the much stronger quantitative result, which we believe to be essentially best possible. The present work was completed independently, and
the proof is different.}. 
The formulation in the 1999 problem collection \emph{Some of Paul's
favorite problems} asks whether
$\limsup_{x\to\infty}|S_f(x)|/\sqrt{x}=+\infty$ almost surely, while the
current online formulation, recorded as Erd\H{o}s Problem~\#1144, asks the
stronger one-sided question whether
$\limsup_{x\to\infty}S_f(x)/\sqrt{x}=+\infty$ almost surely
\cite{PaulsProblems1999,Bloom1144}. 
Our main theorem gives a much stronger quantitative
two-sided strengthening of both formulations.

\begin{theorem}\label{thm:main}
Let $f$ be an extended Rademacher RMF. Then for any fixed deterministic function $V(x)\to\infty$, we have,  almost surely 
\begin{equation}\label{eq:main-theorem-positive}
 S_f(x)=\Omega_{\pm}\bigg(
 \frac{\sqrt{x}(\log\log x)^{1/4}}{V(x)}\bigg).
\end{equation}
Here \(A(x)=\Omega_{\pm}(B(x))\) means that
\(
\limsup_{x\to\infty}{A(x)}/{B(x)}>0\) and \(\liminf_{x\to\infty}{A(x)}/{B(x)}<0 \) for an eventually positive function \(B(x)\).
\end{theorem}

In particular, the positive limsup assertion implies 
$
 \limsup_{x\to\infty}{S_f(x)}/{\sqrt{x}}=+\infty,
$
answering Erd\H{o}s Problem~\#1144 affirmatively in the stronger form.

\begin{remark}
    We note that Hardy and Klurman \cite{HardyKlurman2026} independently obtained the same result on the large fluctuation of extended Rademacher RMF as in Theorem \ref{thm:main}.
\end{remark}

\begin{remark}
It is worth pointing out that our argument may work for a more general model than the extended Rademacher RMF considered here.  What is essentially used is the Euler product factorisation
\[
\sum_{\substack{n\geq1\\P^+(n)\leq X}}\frac{f(n)}{n^s}=F_X(s)Z_X(2s)R_X(s),
\]
where \(F_X(s)=\prod_{p\le X}(1+\varepsilon_p p^{-s})\) is the
squarefree Rademacher Euler product, $Z_X(s)=\prod_{p\leq X}(1-p^{-s})^{-1}$ and \(R_X\) is a ``well-behaved" 
factor (in our case $R_X(s)\equiv1$). 
Under some suitable moment and  regularity assumptions
on \(R_X\), the same argument  applies more generally. 
For instance, the \(m\)-free  extended 
Rademacher RMF has  factorisation
\[
 F_X(s)Z_X(2s)
 \prod_{p\le X}\left(1-\varepsilon_p^m p^{-ms}\right).
\]
For \(m\ge3\) the last product is regular on \(\Re s=1/2\), whereas for
\(m=2\) it cancels \(Z_X(2s)\) completely, giving the usual  Rademacher
RMF. Actually, by combining  \citet{Harper2023Large}'s original argument with our conditional symmetry
observation used in Section~\ref{sec:two-sided}, one can obtain the Rademacher RMF analogues
of Theorems~\ref{thm:main} and~\ref{thm:one-scale}. 

Another related family is obtained by taking \(f(p)\) uniformly among the
\(k\)-th roots of unity and extending completely multiplicatively. Its Euler product decomposition contains the deterministic factor
\(Z_X(ks)\).  Thus \(k=2\) is the critical case in which the associated
zeta singularity occurs at \(s=1/2\), while for \(k\ge3\) it lies to the
left of the critical line; as \(k\to\infty\),  the uniform distribution on the \(k\)-th roots of unity converges weakly to the uniform distribution on the unit circle, corresponding
to the Steinhaus RMF.  We do not pursue the details of these extensions here.
\end{remark}

\begin{remark}
    We also remark that our bound should be sharp up to $(\log \log x)^{o(1)}$. One might be able to prove the corresponding upper bound by combining some techniques here and arguments of \citet{Caich2023} with the recent developments \cite{DurkanPearceCrump2026,Verreault2026}. 

It is worth noting that, 
very recently, Atherfold, Gerspach and Hamdan
\cite{AtherfoldGerspachHamdan2026} announced their results on low moments of extended Rademacher RMFs.
They show that the typical size can behave very differently with the additional \textit{completely multiplicative} property, compared with the usual Rademacher case. While our result 
shows that its influence on large fluctuations is more limited.
\end{remark}

Following the general strategy of \citet{Harper2023Large}, we first prove
a conditional fluctuation estimate for
\(X^{8/7}\leq x\leq X^{4/3}\). For $X\geq3$ let
\[
 \mathscr F_X=\sigma(f(p):p\leq X),\qquad
 \p_X(\,\cdot\,)=\p(\,\cdot\mid\mathscr F_X).
\]
Thus $\mathscr F_X$ records the small prime variables, while conditional on
$\mathscr F_X$ the variables $(f(p))_{p>X}$ remain independent Rademacher
variables. Using the conditional Borel--Cantelli lemma (see Lemma \ref{lem:conditional-borel-cantelli} below),
Theorem~\ref{thm:main} follows from the following result.

\begin{theorem}
\label{thm:one-scale}
Let $f$ be an extended Rademacher RMF. There is an absolute constant $W_0>0$ such that  for all sufficiently large $X$ and all $W_0\leq W\leq 1/(30)\log\log\log X$, we have
\[
\p_X\bigg(
\max_{X^{8/7}\leq x\leq X^{4/3}}
\frac{+1}{\sqrt{x}}\sum_{n\leq x}f(n)
\geq
\frac{(\log\log X)^{1/4}}{e^{1.2W}}
\bigg)
\geq \frac13
\]
with probability  \( \ge1-O(e^{-W/10})\). The same result with $+1$ replaced by $-1$ holds.
\end{theorem}

\begin{proof}[Proof of Theorem~\ref{thm:main} assuming
Theorem~\ref{thm:one-scale}]
We prove the positive assertion; the negative assertion follows in exactly
the same way.

Since $V(x)\to\infty$, we may choose a deterministic increasing sequence
$X_\ell\to\infty$ such that, for every sufficiently large $\ell$,
\[
X_\ell^{4/3}<X_{\ell+1},\quad
W_0\leq W=20\log\ell\leq \frac1{30}\log\log\log X_\ell,
\]
and
$
V(x)\geq \ell^{24}$ for every $x\geq X_\ell^{8/7}.
$

Let
\[
\mathcal E_\ell=
\bigg\{
\max_{X_\ell^{8/7}\leq x\leq X_\ell^{4/3}}
\frac1{\sqrt{x}}\sum_{n\leq x}f(n)
\geq
\frac{(\log\log X_\ell)^{1/4}}{\ell^{24}}
\bigg\}.
\]
By Theorem~\ref{thm:one-scale},  with $W=20\log\ell$, we have 
\[
\p\big(
\p(\mathcal E_\ell\mid\mathscr F_{X_\ell})<1/3
\big)
\ll e^{-2\log\ell}
=\ell^{-2},
\]
where the outer probability is taken over all realizations of the small primes \((f(p))_{p\leq X_\ell}\).
Since $\sum_{\ell\geq2}\ell^{-2}<\infty$, the first Borel--Cantelli lemma
shows that almost surely \(\p(\mathcal E_\ell\mid\mathscr F_{X_\ell})\geq 1/3
\) for every sufficiently large $\ell$.
The condition $X_\ell^{4/3}<X_{\ell+1}$ implies that
$\mathcal E_\ell\in\mathscr F_{X_{\ell+1}}$.
Thus,
\[
\sum_{\ell\geq2}
\p(\mathcal E_\ell\mid\mathscr F_{X_\ell})
=\infty
\quad\text{almost surely}.
\]

Applying Lemma~\ref{lem:conditional-borel-cantelli} to the filtration
$(\mathscr F_{X_{\ell+1}})_{\ell\geq2}$ shows that
$\mathcal E_\ell$ occurs infinitely often almost surely.  Every occurrence of $\mathcal E_\ell$ therefore provides a point satisfying \eqref{eq:main-theorem-positive},  this completes the proof.
\end{proof}

Note that here we use both the first and the conditional
Borel--Cantelli lemmas. This differs from Harper's deduction
\cite{Harper2023Large}, which uses only the first Borel--Cantelli lemma.
In both arguments, however, no independence assumption is required.
\subsection{An application to sign changes}

 We  recall some  known results on sign changes in the two Rademacher
RMFs.
For the Rademacher RMF,
\citet{AymoneHeapZhao2023} proved almost surely infinitely many sign changes
for the unweighted partial sums, and \citet{Aymone2024} later proved the
corresponding result for weighted sums
\(\sum_{n\le x} g(n)/\sqrt n\).
Quantitative almost sure lower bounds for the number of sign changes were subsequently
obtained by \citet{GeisHiary2025} and improved by
\citet{KlurmanLamzouriMunsch2024}; see also \citet{Aymone2025Average} for an
average lower bound.
For the extended Rademacher RMF, \citet{Aymone2024} proved almost
surely infinitely many sign changes for the unweighted partial sums, while
\citet{AngeloXu2024} settled the weighted case
\(\sum_{n\le x} f(n)/\sqrt n\), and also studied the probability that the partial sums remain nonnegative after conditioning on \(f(p)=1\) for all sufficiently small primes, answering a question {of \citet{Kucheriaviy2025}}. Note that, however, no quantitative almost sure lower bounds on the number of sign changes are currently known for the extended Rademacher RMF.
 Our Corollary~\ref{cor:sign-changes} provides the first such quantitative lower bound.
 
For $h\in\{f,g\}$, let $N_h(x)$ denote the number of
sign changes of the partial sums $\sum_{n\leq y}h(n)$ when  $y \in [1,x]$. 
The conditional form of Theorem~\ref{thm:one-scale} has a simple
consequence for sign changes. The point is that, on each of a sequence
of well-separated logarithmic scales, the theorem gives a fixed positive
conditional probability of producing a value with either choice of
sign. Choosing the sign opposite to the last nonzero partial sum on the
current scale then forces a new sign change.

\begin{corollary}\label{cor:sign-changes} Let $h\in\{f,g\}$. We have, almost surely for all sufficiently large $x$,
\[
 N_h(x)\gg\log\log x.
\]
\end{corollary}

For the Rademacher RMF this improves the previous almost sure lower
bound $\gg \log\log x/\log\log\log\log x$ of
\citet[Corollary~1.7]{KlurmanLamzouriMunsch2024}. 

\begin{proof}[Proof of Corollary \ref{cor:sign-changes} assuming Theorem \ref{thm:one-scale}]
Let $X_\ell=\exp(2^\ell)$ and $W_\ell=(\log\ell)/60$. For all
sufficiently large $\ell$, these parameters lie in the range of
Theorem~\ref{thm:one-scale}, and $X_\ell^{4/3}<X_{\ell+1}=X_{\ell}^2$.
{Let $\Omega_\ell$ be the $\mathscr F_{X_\ell}$-measurable event on which
the two conditional estimates in Theorem~\ref{thm:one-scale} hold
simultaneously. Then $\p(\Omega_\ell^c)\ll \ell^{-1/600}$ and, }
on $\Omega_\ell$, either
sign occurs somewhere in $[X_\ell^{8/7},X_\ell^{4/3}]$ with
conditional probability at least $1/3$.

Let $\mathcal E_\ell:=\{N_f(X_{\ell+1})-N_f(X_{\ell})\ge1\}$ be the event that $S_f(y)$ has a new sign change
between $X_\ell$ and $X_{\ell+1}$. The sign of the last nonzero value
of $S_f(y)$ with $y\leq X_\ell$ is $\mathscr F_{X_\ell}$-measurable.
We may therefore apply Theorem~\ref{thm:one-scale} with the opposite
sign, and hence, if
$q_\ell=\p(\mathcal E_\ell\mid\mathscr F_{X_\ell})$, then
$q_\ell\geq \frac13\mathbf 1_{\Omega_\ell}$.

Since $\mathcal E_\ell\in\mathscr F_{X_{\ell+1}}$, the variables
\(
 D_\ell
 :=
 \mathbf 1_{\mathcal E_\ell}-q_\ell
 =
 \mathbf 1_{\mathcal E_\ell}
 -
 \E(\mathbf 1_{\mathcal E_\ell}\mid\mathscr F_{X_\ell})
\)
form a martingale-difference sequence with respect to the filtration
$(\mathscr F_{X_{\ell+1}})_\ell$. Indeed,
$\E(D_\ell\mid\mathscr F_{X_\ell})=0$, and for $j<\ell$ the variable
$D_j$ is $\mathscr F_{X_\ell}$-measurable, whence
$\E(D_jD_\ell)=0$. Thus for any fixed sufficiently large $\ell_0$,
\[
 \E\bigg|
 \sum_{\ell_0\leq\ell\leq L}D_\ell
 \bigg|^2
 =
 \sum_{\ell_0\leq\ell\leq L}\E|D_\ell|^2
 \leq L.
\]
Moreover,
$\E\sum_{\ell_0\leq\ell\leq L}\mathbf 1_{\Omega_\ell^c} \ll L^{1-1/600}$. 
Since $q_\ell\geq\frac13\mathbf 1_{\Omega_\ell}$, outside the event
that $\Omega_\ell$ fails for a positive proportion of
$\ell\in[\ell_0,L]$, we have
$\sum_{\ell_0\leq\ell\leq L}q_\ell\gg L$. On the other hand, we have \( \sum_{\ell_0\leq\ell\leq L}\mathbf 1_{\mathcal E_\ell} \leq N_f(X_{L+1}).
\)
Thus, if $N_f(X_{L+1})\leq cL$ for some sufficiently small absolute constant $c>0$, then this forces either $\Omega_\ell$ fails for a
positive proportion of $\ell\in[\ell_0,L]$ or
$\left|\sum_{\ell_0\leq\ell\leq L}D_\ell\right|\gg L$. 
Hence
\[
\begin{aligned}
 \p\bigl(N_f(X_{L+1})< cL\bigr)
 &\leq
 \p\bigg(
   \sum_{\ell_0\leq\ell\leq L}\mathbf 1_{\Omega_\ell^c}\gg L
 \bigg)
 +
 \p\bigg(
   \bigg|\sum_{\ell_0\leq\ell\leq L}D_\ell\bigg|\gg L
 \bigg) \\
 &\ll
 \frac1L\,
 \E\sum_{\ell_0\leq\ell\leq L}\mathbf 1_{\Omega_\ell^c}
 +
 \frac1{L^2}\,
 \E\bigg|
   \sum_{\ell_0\leq\ell\leq L}D_\ell
 \bigg|^2 \\
 &\ll L^{-1/600}+L^{-1}.
\end{aligned}
\]

Taking $L=2^m$ and applying the first Borel--Cantelli lemma, we obtain
almost surely
$N_f(X_{2^m+1})\gg2^m$ for all sufficiently large $m$.
If $X_{2^m+1}\leq x<X_{2^{m+1}+1}$, then
$N_f(x)\gg 2^m \gg\log\log x$ for all sufficiently large $x$. 

Finally the Rademacher RMF case  follows in exactly the same way from the squarefree
analogue of Theorem~\ref{thm:one-scale} noted above.
\end{proof}

\begin{remark}
The proof of Corollary~\ref{cor:sign-changes} uses only the sign
information in Theorem~\ref{thm:one-scale}; the size of the fluctuation
plays no role. Moreover, even if two large values of opposite signs occur, the partial sums between them may change sign many
times, while our argument counts
only once. We believe there is a great room to gain more information by modifying our argument in the proof of
Theorem~\ref{thm:one-scale}, rather than using it only as a black box.
\end{remark}

\subsection{Outline of the proof of Theorem \ref{thm:one-scale}}
Fix a large integer $X$. Suppose that \(x<X^2\). Every integer \(n\leq x\) has at most one prime
factor exceeding \(X\), and using multiplicativity we can break up the sums as
\[
 \frac{1}{\sqrt x}\sum_{n\leq x}f(n)
 =
 \frac{1}{\sqrt x}
 \sum_{\substack{n\leq x\\P^+(n)\leq X}}f(n)
 +
 \frac{1}{\sqrt x}\sum_{X<p\leq x}f(p)
 \sum_{\substack{m\leq x/p\\P^+(m)\leq X}}f(m).
\]
After conditioning on \((f(p))_{p\leq X}\), the first term is fixed, whereas
the second is a centred linear form in the independent Rademacher variables
\((f(p))_{X<p\le x}\).

The first term is one of the points at which the extended Rademacher
RMF differs from the Rademacher RMF. The exact convolution identity
$
 f=(\mu^2f)*\mathbf 1_{\square}
$
gives, for \(x<X^2\),
\begin{equation}\label{eq:convolution-intro}
 \frac{1}{\sqrt x}
 \sum_{\substack{n\leq x\\P^+(n)\leq X}}f(n)
 =
 \sum_{\substack{a\leq x\\P^+(a)\leq X}}\frac{g(a)}{\sqrt a}
 -
 \frac{1}{\sqrt x}
 \sum_{\substack{a\leq x\\P^+(a)\leq X}}
 g(a)\left\{\sqrt{\frac xa}\right\}.
\end{equation}

The variables \(g(a)\), with \(a\) squarefree, form an orthonormal system,
so the second term on the right has second moment at most \(1\). The first
term is handled by Parseval's identity \eqref{eq:parseval-dirichlet}. With a horizontal shift
\(\sigma=1/\log X\), the relevant identity is
\[
\begin{aligned}
 &\int_1^\infty
 \Bigg|
 \sum_{\substack{a\leq x\\P^+(a)\leq X}}
 \frac{g(a)}{\sqrt a}
 \Bigg|^2
 \frac{\mathrm dx}{x^{1+2\sigma}} =
 \frac{1}{2\pi}\int_{\mathbb R}
 \frac{1}{\sigma^2+t^2}
 \left|
 \displaystyle\prod_{p\leq X}
 \left(1+\frac{f(p)}
 {p^{1/2+\sigma+it}}\right)
 \right|^2
 \,\mathrm dt .
\end{aligned}
\]
The multiplicative chaos estimates from
\cite{Harper2020Low,Harper2023Large} control the Euler product  on intervals
of length \(1\) in the frequency variable \(t\). The kernel is singular near \(t=0\), and
this range requires a separate argument. We factor the Euler product at primes up to \(\exp(1/|t|)\), compare its
value at \(t\) with its value at \(t=0\), and control the latter using a
standard large-deviation estimate for weighted Rademacher sums.
This gives
\[
 \mathbb P\Bigg(
 \frac{1}{\log X}
 \int_{X^{8/7}}^{X^{4/3}}
 \Bigg|
 \frac{1}{\sqrt x}
 \sum_{\substack{n\leq x\\P^+(n)\leq X}}f(n)
 \Bigg|^2\frac{\mathrm dx}{x}
 >
 (\log\log X)^{1/3}
 \Bigg)
 \ll(\log\log X)^{-1/200}.
\]
Thus no pointwise estimate for either term in
\eqref{eq:convolution-intro} is needed.

This is another point where the extended Rademacher RMF differs
from the Rademacher RMF.  In the latter, orthogonality gives
an easy second moment bound, showing that the \(X\)-smooth contribution is small at most points of the fixed logarithmic grid $X^{8/7}e^{2\pi r}$, as in \cite[Section~3.1]{Harper2023Large}.  Here the square contributions destroy this orthogonality, the corresponding easy second moment estimate is no longer available. We therefore turn to prove a logarithmic mean square estimate over \(X^{8/7}\le x\le X^{4/3}\).  To transfer this continuous integral estimate to a logarithmically
spaced grid, we introduce an additional  parameter \(\theta \in[0, 2\pi) \).

We sample at logarithmically spaced points
\[
 x=X^{8/7}e^{\theta+2\pi r},
 \qquad 0\leq\theta<2\pi,
\]
where \(r\) ranges over the integers for which \(4x\leq X^{4/3}\).
There are \(\asymp\log X\) such points.
Averaging over the common
shift \(\theta\), the preceding logarithmic mean square estimate shows that,
for all but a proportion
$
 O\!\left(e^{1.2W}(\log\log X)^{-1/12}\right)
$
of the shifts, all but the same proportion of the indices \(r\) satisfy
\[
 \Bigg|
 \frac{1}{\sqrt y}
 \sum_{\substack{n\leq y\\P^+(n)\leq X}}f(n)
 \Bigg|
 \leq {A}:=e^{-1.2W}(\log\log X)^{1/4}
\]
both at $y\in\{x,4x\}$. It is important here that the same shift $\theta$  is used for every $r$: this preserves the logarithmic spacing $\log x_{r,\theta}-\log x_{s,\theta}=2\pi(r-s)$
needed in the covariance argument.

We next consider the contribution from primes exceeding \(X\). Rather than
study its value at the point $x$, we take the difference between the
normalised values at \(x\) and \(4x\), which we also call the large prime increment at \(x\), namely,
\begin{align}\label{eq:large-part-decom}
 \frac{1}{\sqrt x}\sum_{X<p\leq4x}f(p)
 \Bigg(
 \frac12
 \sum_{\substack{m\leq4x/p\\P^+(m)\leq X}}f(m)
 -
 \sum_{\substack{m\leq x/p\\P^+(m)\leq X}}f(m)
 \Bigg).
\end{align}
Viewed as a function of \(x/p\), the expression in parentheses has
Mellin transform
\begin{equation}\label{eq:multiplier-intro}
 \frac{4^{s-1/2}-1}{s}
 \prod_{p\leq X}\left(1+\frac{f(p)}{p^s}\right)
 \prod_{p\leq X}\left(1-\frac{1}{p^{2s}}\right)^{-1}.
\end{equation}
The second product is the finite analogue of the factor \(\zeta(2s)\) in
\eqref{eq:euler-factorisation-intro}. Near \(s=1/2\) it has singular
growth as \(X\to\infty\), but \(4^{s-1/2}-1\) vanishes to first order there.
The product
\(
 \left(4^{s-1/2}-1\right)
 \prod_{p\leq X}\left(1-{p^{-2s}}\right)^{-1}
\)
is therefore  regular (uniformly bounded) near \(s=1/2\). On a fixed frequency
interval, for example \(1/3\leq |t|\leq1/2\), it is also bounded away from
zero. The factor $4^{s-1/2}-1$  thus removes the low frequency singularity while keeping the fixed frequency contribution needed for the lower bound.

\begin{remark}
The use of an auxiliary factor vanishing at a zeta type singularity is
classical in spirit. 
A closely related cancellation occurs in the work of Soundararajan \cite[Section~5.3]{Soundararajan2000}, where the factor
\(
 \bigl(1-2^{1-2s}\bigr)\zeta(2s)
\)
appears, with the zero at \(s=1/2\) canceling the pole of
\(\zeta(2s)\).  This is essentially the same as ours.
In the present setting, our canceling factor is not inserted ad hoc: it arises naturally in \eqref{eq:multiplier-intro} from \eqref{eq:large-part-decom}.
\end{remark}

Conditioning on the values \(f(p)\), \(p\leq X\), the large prime
part \eqref{eq:large-part-decom} has mean zero, and its conditional variance is given by the corresponding square sum over \(X<p\leq4x\).  
A short interval prime estimate bounds this square sum from below by a
weighted integral. Parseval's identity
\eqref{eq:parseval-dirichlet}, together with
\eqref{eq:multiplier-intro}, then turns this integral into an Euler product
integral on the Mellin line.
Restricting the \(t\)-integral to \(1/3\leq |t|\leq1/2\), the
multiplicative chaos lower bound from \cite{Harper2023Large} gives, with
probability \(1-O(e^{-0.1W})\),
\[
 \frac{1}{x}\sum_{X<p\leq4x}
 \Bigg(
 \frac12
 \sum_{\substack{m\leq4x/p\\P^+(m)\leq X}}f(m)
 -
 \sum_{\substack{m\leq x/p\\P^+(m)\leq X}}f(m)
 \Bigg)^2
 \gg
 \frac{e^{-2.2W}}{\sqrt{\log\log X}},
\]
simultaneously for all \(x_{r,\theta}\) in the shifted grid.

This variance estimate also explains the exponent \(1/4\). If one had
\((\log X)^{0.19}\) independent Gaussian variables with variances of the
preceding order, then their maximum would naturally have size
\[
 \sqrt{\frac{e^{-2.2W}}{\sqrt{\log\log X}}}\,
 \sqrt{\log\big((\log X)^{0.19}\big)}
 \asymp e^{-1.1W}(\log\log X)^{1/4}.
\]
This is larger than the target
\({A}=e^{-1.2W}(\log\log X)^{1/4}\), since \(W\ge W_0\) large enough. The main remaining
problem is therefore to find sufficiently many points whose conditional
covariances are small.

For this we prove a weighted version of Harper's high mixed moment
estimate. Outside an \(\mathscr F_X\)-measurable event of small
probability, uniformly for every common shift \(\theta\), each increment
has conditional covariance exceeding
\((\log\log X)^{-0.59}\) with at most \((\log X)^{0.8}\) of the other
increments.
Since all but a small proportion of the original \(\asymp\log X\) grid
points have small \(X\)-smooth contribution, a greedy selection among
these points leaves at least \((\log X)^{0.19}\) points with all pairwise
covariances below this threshold.
Dividing by the variance lower bound, the corresponding
correlations are at most
\[
 \frac{(\log\log X)^{-0.59}}
 {e^{-2.2W}(\log\log X)^{-1/2}}
 =
 e^{2.2W}(\log\log X)^{-0.09}=o(1).
\]

We briefly indicate the high moment argument. After Perron inversion, the
conditional covariance is represented by a double integral involving the
Rademacher Euler product in \eqref{eq:multiplier-intro}, a short prime sum,
and the deterministic multiplier \(4^{s-1/2}-1\).
Taking a \(2k\)-th moment, where
\(
 k=\big\lfloor
 \frac{\log\log X}{5000\log\log\log X}
 \big\rfloor,
\)
and expanding introduces frequencies \(t_1,\ldots,t_{2k}\). Summation over
the logarithmic grid produces the geometric sum
\[
 \sum_r
 e^{-2\pi i r\Xi}
 \ll
 \min\bigg(
 \log X,\,
 \frac{1}{
 \left\| \Xi
 \right\|_{\mathbb R/\mathbb Z}}
 \bigg),
\]
where $ \Xi:=t_1+\cdots+t_k-t_{k+1}-\cdots-t_{2k}$.
Away from the integers this already gives the required saving. On the near-integer range, we use the strong barrier condition and, as in
\cite[Sections~3.3--3.4]{Harper2023Large}, order the frequencies and decompose according to the sizes of the gaps between consecutive
frequencies. Roughly speaking, when many consecutive frequencies are close together, the barrier gives a large saving, while when many of the gaps are fairly large, the near-integer constraint gives the required
volume saving.

The main point to keep in mind when adapting Harper's argument is that every
Euler product integral carries the deterministic multiplier from
\eqref{eq:multiplier-intro}. On the relevant frequency range, this
multiplier is \(O((\log\log\log X)^2)\), so it contributes only a harmless
additional loss in the \(2k\)-th moment estimate. The resulting weighted
moment estimate gives the required sparsity of the large conditional
covariances.

Conditional on the primes up to \(X\), the selected increments are
linear forms in independent Rademacher variables. We use Harper's
multivariate normal approximation to compare their maximum with that of a
centred Gaussian vector having the same covariance matrix; the required
error terms are controlled by the fixed third- and fourth-moment estimates
from Lemma~\ref{lem:complete-moments}. After standardisation, the Gaussian
variables have uniformly small pairwise correlations. Harper's normal
comparison result, followed by the normal approximation, then shows that
the maximum of the selected increments exceeds
\(
 {4A}=4e^{-1.2W}(\log\log X)^{1/4}
\)
with conditional probability \(1-o(1)\).

Such an increment implies that the large prime contributions at the selected endpoints have a range of at least
\({4A}\). Conditional on the primes up to \(X\),
these contributions form a centrally symmetric vector: simultaneously
replacing \(f(p)\) by \(-f(p)\) for \(X<p\leq X^{4/3}\) changes the vector
to its negative. Hence, with conditional probability \(1-o(1)\), either
the maximum is at least
\({2A}\) or the minimum is at most
\({-2A}\). By symmetry, each of these two events
has conditional probability at least \(1/2-o(1)\).
By our choice of the selected points, the \(X\)-smooth contribution at
both corresponding endpoints has absolute value at most
\({A}\). Adding this contribution back therefore
gives a large value of either sign and completes the proof of
Theorem~\ref{thm:one-scale}.

\subsection{Organisation}The paper is organised as follows.
Section~\ref{sec:preliminaries} records the large prime decomposition,
the square divisor identity, and the probabilistic inputs.
Section~\ref{sec:old-primes} proves the logarithmic mean square estimate
for the \(X\)-smooth contribution.
Section~\ref{sec:variance} introduces the  multiplier and
establishes the lower bound for the conditional variance.
Section~\ref{sec:covariances} proves the weighted high moment estimate and
the sparsity of large conditional covariances.
Section~\ref{sec:fresh-primes} establishes the conditional Gaussian approximation, and Section~\ref{sec:common-shift} uses a common logarithmic shift to transfer the logarithmic mean square estimate to most points of the grid.
Finally, Section~\ref{sec:two-sided} combines these ingredients and
completes the proof of Theorem~\ref{thm:one-scale}.

\subsection*{Acknowledgement} The work was initiated during the Universality in Number Theory Program at CRM Montreal 2026. The authors thank the hospitality of CRM and the organizers. We thank Adam Harper for helpful discussions on this problem. 
H.G. would like to express sincere gratitude to Professors Andrew Granville and Jianya Liu for their guidance and support during his Ph.D. studies. He also thanks  financial support from the China Scholarship Council (CSC).

\section{Preliminaries}\label{sec:preliminaries}

Throughout the paper $p$ denotes a prime, $P^+(n)$ is the largest prime
factor of $n$, and $P^+(1)=1$.  Constants implicit in $O$- and
$\ll$-notation are absolute unless a dependence is displayed.  We use the
same probability space for
\[
 f(n)=\prod_{p^\nu\Vert n}\epsilon_p^\nu
 \quad\text{and}\quad
 g(n)=\mu^2(n)f(n).
\]
The function $g(n)$ is the usual Rademacher RMF
supported on squarefree integers. Since $\epsilon_p^2=1$, one has the exact
identity
\begin{equation}\label{eq:square-convolution}
 f(n)=\sum_{d^2\mid n}g(n/d^2).
\end{equation}
Exactly one summand on the right is non-zero.

For $X\geq3$, let $\mathscr F_X=\sigma(\epsilon_p:p\leq X)$. We write
$\p_X(\,\cdot\,)=\p(\,\cdot\mid\mathscr F_X)$ and
$\E_X(\,\cdot\,)=\E(\,\cdot\mid\mathscr F_X)$, and denote the
corresponding conditional variance and covariance by $\var_X(\,\cdot\,)$ and
$\cov_X(\,\cdot\,)$, respectively. We also write
\begin{equation}\label{eq:def-H-C}
 C_X(y)=\frac1{\sqrt y}
 \sum_{\substack{n\leq y\\P^+(n)\leq X}}f(n).
\end{equation}
An empty sum is understood when $y<1$.  The Rademacher Euler product will
be denoted by
\begin{equation}\label{eq:def-squarefree-product}
 F_X(s)=\prod_{p\leq X}(1+\epsilon_pp^{-s}).
\end{equation}
We also introduce the deterministic Euler product
\begin{equation}\label{eq:def-zeta-product}
 Z_X(s)=\prod_{p\leq X}(1-p^{-s})^{-1}.
\end{equation}
The Dirichlet series of $f$ restricted to the $X$-smooth integers factors
as
\begin{equation}\label{eq:finite-euler-factorisation}
 \sum_{\substack{n\geq1\\P^+(n)\leq X}}\frac{f(n)}{n^s}
 =
 \prod_{p\leq X}(1-\epsilon_pp^{-s})^{-1}
 =
 F_X(s)Z_X(2s),
 \qquad \mathrm{Re}\,s>0.
\end{equation}

For $1\leq x<X^2$, since every \(n<X^2\) has at most one prime factor exceeding \(X\), we have
\begin{equation}\label{eq:fresh-decomposition}
 \frac{S_f(x)}{\sqrt x}
 =
 C_X(x)+U_X(x),
 \qquad
 U_X(x)=
 \sum_{X<p\leq x}\frac{\epsilon_p}{\sqrt p}\,C_X(x/p).
\end{equation}
Conditional on $\mathscr F_X$, the coefficients in $U_X(x)$ are fixed,
whereas $(\epsilon_p)_{p>X}$ are independent Rademacher variables.

We shall compare the endpoints $x$ and $4x$. Accordingly, set
\begin{equation}\label{eq:def-B}
 B_X(x)=
 \frac12\sum_{\substack{n\leq4x\\P^+(n)\leq X}}f(n)
 -
 \sum_{\substack{n\leq x\\P^+(n)\leq X}}f(n).
\end{equation}
If $4x<X^2$, we define the large prime increment by
\begin{equation}\label{eq:fresh-increment}
 Y_x
 \coloneqq U_X(4x)-U_X(x)
 =
 \frac1{\sqrt x}\sum_{X<p\leq4x}\epsilon_pB_X(x/p).
\end{equation}

Only fixed moments are needed for the normal approximation. The following crude  bounds are sufficient for our purpose.

\begin{lemma}[Moments of the $X$-smooth partial sums]\label{lem:complete-moments}
Let $0\leq y<X$. For each $r\in\{2,3,4\}$ there is $A_r>0$ such that
\begin{equation}\label{eq:complete-moment-bound}
 \E\ \bigg|
 \sum_{\substack{n\leq y\\P^+(n)\leq X}}f(n)
 \bigg|^r
 \ll_r(1+y)^{r/2}\log^{A_r}(2+y).
\end{equation}
The same conclusion holds for $B_X(y)$ whenever $4y<X$.
\end{lemma}
\begin{proof}
In the stated range every integer not exceeding $y$ is automatically
$X$-smooth. Using \eqref{eq:square-convolution}, we therefore have
\[
 \sum_{\substack{n\leq y\\P^+(n)\leq X}}f(n)
 =
 \sum_{d\leq\sqrt y}\sum_{m\leq y/d^2}g(m).
\]
For a random variable $Y$ and $q\geq1$, write $ \|Y\|_q=(\E|Y|^q)^{1/q}.$
By the hypercontractive inequality for Rademacher RMFs (see, for example, \cite[Probability Result~2.3]{Harper2019High}),
for \(k\in\{1,2\}\),
\[
 \E\bigg|\sum_{m\leq z}g(m)\bigg|^{2k}
 \leq
 \bigg(
 \sum_{m\leq z}\mu^2(m)\tau_{2k-1}(m)
 \bigg)^k,
\]
where $\tau_j$ is the $j$-fold divisor function.  
It follows that, by the standard upper bound  $ \sum_{m\leq z}\tau_{k}(m)
 \ll_k z\log^{k-1}(2+z)$,
\[
 \bigg\|\sum_{m\leq z}g(m)\bigg\|_{2k}
 \ll_k\sqrt z\log^{k-1}(2+z)
 \qquad(k=1,2),
\]
Minkowski's inequality now gives us
\begin{align*}
 \bigg\|
 \sum_{\substack{n\leq y\\P^+(n)\leq X}}f(n)
 \bigg\|_{2k}
 &\leq
 \sum_{d\leq\sqrt y}
 \bigg\|\sum_{m\leq y/d^2}g(m)\bigg\|_{2k} 
 \ll_k
 \sqrt y\log^{k-1}(2+y)
 \sum_{d\leq\sqrt y}\frac1d 
 \ll_k
 \sqrt y\log^k(2+y).
\end{align*}
This proves the cases $r=2$ and $r=4$, while the case $r=3$ follows from Cauchy--Schwarz,
\[
 \E\ \bigg|
 \sum_{\substack{n\leq y\\P^+(n)\leq X}}f(n)
 \bigg|^3
 \leq
 \bigg(
 \E\bigg|
 \sum_{\substack{n\leq y\\P^+(n)\leq X}}f(n)
 \bigg|^2
 \bigg)^{1/2}
 \bigg(
 \E\bigg|
 \sum_{\substack{n\leq y\\P^+(n)\leq X}}f(n)
 \bigg|^4
 \bigg)^{1/2}.
\]
If $4y<X$, the same bounds apply to the two partial sums in
\eqref{eq:def-B}; the assertion for $B_X(y)$ follows from the triangle
inequality.
\end{proof}

We shall repeatedly use the following form of Parseval's identity.  If
$A(s)=\sum_{n\geq1}a_nn^{-s}$ converges absolutely for $\mathrm{Re}\,s>c$, then, whenever both sides converge,
\begin{equation}\label{eq:parseval-dirichlet}
 \int_0^\infty\Big|\sum_{n\leq x}a_n\Big|^2x^{-1-2\sigma}\,\mathrm{d}x
 =\frac1{2\pi}\int_\R
   \left|\frac{A(\sigma+it)}{\sigma+it}\right|^2\,\mathrm{d}t.
\end{equation}
This is Parseval's identity for Dirichlet series; see
\cite[Section~5.1, especially (5.26)]{MontgomeryVaughan2007}.
All applications below are made after a positive horizontal shift, where
the required convergence is immediate.

We shall also use the following L\'evy's extension of the second
Borel--Cantelli lemma; see \cite[Theorem~4.3.4]{Durrett2019}.

\begin{lemma}[Conditional Borel--Cantelli lemma]
\label{lem:conditional-borel-cantelli}
Let $(\mathscr F_\ell)_{\ell\geq0}$ be a filtration, and let
$\mathcal E_\ell\in\mathscr F_\ell$ for every $\ell\geq1$. Then, almost surely the following events are the same
\[
 \{\mathcal E_\ell\ \text{infinitely often}\}
 =
 \bigg\{
 \sum_{\ell\geq1}
 \p(\mathcal E_\ell\mid\mathscr F_{\ell-1})
 =\infty
 \bigg\}.
\]
\end{lemma}

The multiplicative chaos estimates used below are those proved in
\cite[Sections~2.3--2.4]{Harper2023Large}. We use in particular
Multiplicative Chaos Result~3 for low-moment upper bounds,
Multiplicative Chaos Result~4 for a lower bound on a fixed frequency interval, and the
barrier estimates from Sections~3.3--3.4. All applications are made to
the squarefree Euler product \(F_X\) or to its prime truncations; we do
not apply these results directly to the Euler product \(F_X(s)Z_X(2s)\)
of the extended Rademacher RMF. The relevant parameters are specified at each
application.

We also use Normal Approximation Result~1 and Normal Comparison
Result~1 of \cite{Harper2023Large}. The former is used in
Section~\ref{sec:fresh-primes} to replace the large prime increments by a Gaussian vector with the same covariance matrix, while
the latter is used in the proof of Theorem~\ref{thm:one-scale} to control
the maximum of the resulting weakly correlated Gaussian variables.

\section{The \texorpdfstring{$X$}{X}-smooth contribution}\label{sec:old-primes}

The first issue peculiar to the extended Rademacher RMF is that,
after conditioning on $(f(p))_{p\leq X}$, there remains a non-centred,
$\mathscr F_X$-measurable sum supported on $X$-smooth integers. The identity
$f=(\mu^2f)*\mathbf 1_{\square}$ isolates this term in a useful form.
Throughout this section, $g$ and $C_X$ have the meanings fixed in
\eqref{eq:square-convolution} and \eqref{eq:def-H-C}.
Only the values $X^{8/7}\leq x\leq X^{4/3}$ will be needed. 

For $1\leq x<X^2$, every $d\leq\sqrt x$ is automatically $X$-smooth.
Hence \eqref{eq:square-convolution} gives
\[
 C_X(x)=\frac1{\sqrt x}
 \sum_{\substack{a\leq x\\P^+(a)\leq X}}g(a)
 \left\lfloor\sqrt{\frac xa}\right\rfloor.
\]
Writing the integer part as its argument minus its fractional part, we
obtain
\[
 C_X(x)=M_X(x)-D_X(x),
\]
where
\[
 M_X(x)=\sum_{\substack{a\leq x\\P^+(a)\leq X}}\frac{g(a)}{\sqrt a},
 \qquad
 D_X(x)=\frac{1}{\sqrt x}
 \sum_{\substack{a\leq x\\P^+(a)\leq X}}g(a)
 \left\{\sqrt{\frac xa}\right\}.
\]
The random variables $g(a)$, with $a$ squarefree, form an orthonormal
system. Hence
\[
 \E D_X(x)^2=\frac1x
 \sum_{\substack{a\leq x\\P^+(a)\leq X}}
 \left\{\sqrt{\frac xa}\right\}^{\!2}
 \leq\frac1x\sum_{a\leq x}1\leq1.
\]

We next prove the estimate used to control the $X$-smooth contribution at
most endpoints. Recall the Rademacher Euler product $F_X$ from
\eqref{eq:def-squarefree-product}; equivalently,
$F_X(s)=\sum_{P^+(n)\leq X}g(n)n^{-s}$.
The multiplicative chaos input is the following low-moment
bound.  It is Multiplicative Chaos Result~3 of
\cite{Harper2023Large}; see also \cite[Key Proposition~3 and the discussion
following it]{Harper2020Low}, and the formulation in
\cite[Proposition~2.9]{Atherfold2025}.  For $q=2/3$,
$\sigma=1/\log X$, and $|N|\leq(\log X)^{1000}$, it gives
\begin{equation}\label{eq:mc3-old}
 \E\left(\int_{N-1/2}^{N+1/2}
 \left|F_X\left(\frac12+\sigma+it\right)\right|^2\,\mathrm{d}t
 \right)^q
 \ll \bigl(\log\log(|N|+10)\bigr)^q
 \left(\frac{\log X}{\sqrt{\log\log X}}\right)^q.
\end{equation}

{The weighted integral below is closely related  to the methods used by
Gerspach \cite{Gerspach2022} in his study of low pseudomoments, and by Hardy \cite{Hardy2024Weighted} for weighted Steinhaus RMFs, and, more directly in the Rademacher setting, by Atherfold  \cite{Atherfold2025}.} A related weighted Euler product integral also appears in the recent work of Harper--Soundararajan--Xu \cite[Section~2]{HarperSoundararajanXu2026},
where a short interval Fourier transform also produces a decay factor
there.  In the present setting, however, the Parseval kernel
\((\sigma^2+t^2)^{-1}\) becomes singular at \(t=0\) as
\(\sigma=1/\log X\to0\), so the main additional point is to control
the contribution of the very small frequencies.

\begin{lemma}[A weighted mean square bound for the Rademacher Euler product]
\label{lem:weighted-squarefree-energy}
Let $\sigma=1/\log X$.  Then, for all sufficiently large
$X$,
\begin{equation}\label{eq:weighted-energy-prob}
 \p\left(
 \frac1{\log X}\int_{-\infty}^{\infty}
 \frac{|F_X(1/2+\sigma+it)|^2}{\sigma^2+t^2}\,\mathrm{d}t
 \gg (\log\log X)^{-1/200}\right)\ll (\log\log X)^{-1/200}.
\end{equation}
\end{lemma}

\begin{proof}
Set $T_0=(\log\log X)^{-1/10}$.  We divide the $t$-integral in
\eqref{eq:weighted-energy-prob} into $|t|\leq T_0$, $T_0<|t|\leq1$,
$1<|t|\leq(\log X)^{900}$, and the remaining tail.

We begin with the range containing the singularity. If
$\sigma\leq T\leq T_0$ is a dyadic multiple of $\sigma$, put
$Y=\exp(1/T)$.
For $r=p^{-1/2-\sigma}$ and real $\phi$, direct averaging over a Rademacher
variable $\epsilon$ gives the exact identity
\begin{equation}\label{eq:single-prime-ratio}
 \E_{\epsilon}\left|
 \frac{1+\epsilon r e^{-i\phi}}{1+\epsilon r}
 \right|^2
 =1+\frac{4r^2(1-\cos\phi)}{(1-r^2)^2}.
\end{equation}
Taking the product of \eqref{eq:single-prime-ratio} over $p\leq Y$ and
using independence, we obtain
\[
 \begin{aligned}
 &\E\prod_{p\leq Y}\left|
 \frac{1+f(p)p^{-1/2-\sigma-it}}
 {1+f(p)p^{-1/2-\sigma}}
 \right|^2=
 \prod_{p\leq Y}
 \left(
 1+
 \frac{4p^{-1-2\sigma}(1-\cos(t\log p))}
 {(1-p^{-1-2\sigma})^2}
 \right).
 \end{aligned}
\]
Since $(1-p^{-1-2\sigma})^{-2}\ll1$ uniformly in $p$, taking logarithms
and using $\log(1+v)\leq v$ and $1-\cos v\ll v^2$, we have
\[
 \begin{aligned}
 \log\E\prod_{p\leq Y}\left|
 \frac{1+f(p)p^{-1/2-\sigma-it}}
 {1+f(p)p^{-1/2-\sigma}}
 \right|^2
 &\ll
 t^2\sum_{p\leq Y}\frac{(\log p)^2}{p}\ll t^2(\log Y)^2.
 \end{aligned}
\]
For $|t|\leq2T$ and $\log Y=1/T$, the last expression is $O(1)$.
Consequently,
\begin{equation}\label{eq:short-ratio}
 \E\prod_{p\leq Y}\left|
 \frac{1+f(p)p^{-1/2-\sigma-it}}
 {1+f(p)p^{-1/2-\sigma}}
 \right|^2
 \ll1
\end{equation}
uniformly for $|t|\leq2T$.
Independence and Mertens' estimate also
give
\begin{equation}\label{eq:long-product-second-moment}
 \E\prod_{Y<p\leq X}|1+f(p)p^{-1/2-\sigma-it}|^2
 =\prod_{Y<p\leq X}(1+p^{-1-2\sigma})
 \ll\frac{\log X}{\log Y}=T\log X.
\end{equation}

It remains to control the real-axis factor that was divided out in
\eqref{eq:short-ratio}.  Set $u=\log(1/T)$.  Since
$\sigma\log Y\leq1$, Mertens' estimate and partial summation show uniformly
that
\begin{equation}\label{eq:variance-real-axis}
 \sum_{p\leq Y}p^{-1-2\sigma}=u+O(1).
\end{equation}
Taylor expansion, with an absolutely summable remainder, yields
\begin{equation}\label{eq:real-axis-expansion}
 \log|F_Y(1/2+\sigma)|^2
 =2\sum_{p\leq Y}\frac{f(p)}{p^{1/2+\sigma}}
 -\sum_{p\leq Y}\frac1{p^{1+2\sigma}}+O(1).
\end{equation}
The event \(
 |F_Y(1/2+\sigma)|^2>T^{1/3}=e^{-u/3}
\)
therefore forces
\[
 \sum_{p\leq Y}\frac{f(p)}{p^{1/2+\sigma}}
 >\frac{u}{3}-O(1).
\]
Hoeffding's inequality consequently gives
 \begin{align}\label{eq:real-axis-tail}
 \p\left(|F_Y(1/2+\sigma)|^2>T^{1/3}\right)
 &\leq
 \exp\left(
 -\frac{(u/3-O(1))^2}{2(u+O(1))}
 \right)\ll T^{1/19}
 \end{align}

Let $\Omega_X^{(0)}$ be the barrier event that
$|F_{\exp(1/T)}(1/2+\sigma)|^2\leq T^{1/3}$ for every
$T=2^j\sigma\leq T_0$, where $j\geq0$. These dyadic intervals, together with
$|t|\leq\sigma$, cover $|t|\leq T_0$.  Summing
\eqref{eq:real-axis-tail} over the geometric sequence gives
\begin{equation}\label{eq:good-real-axis}
 \begin{aligned}
 \p\bigl((\Omega_X^{(0)})^c\bigr)
 &\ll
 \sum_{\substack{j\geq0\\2^j\sigma\leq T_0}}
 (2^j\sigma)^{1/19}\ll T_0^{1/19}
 \ll (\log\log X)^{-1/200}.
 \end{aligned}
\end{equation}
We now explain why no independence between $\Omega_X^{(0)}$ and the Euler
product ratio is required. On this event, and for $T\leq|t|\leq2T$, one has
pointwise
\begin{align*}
 \mathbf 1_{\Omega_X^{(0)}}|F_X(1/2+\sigma+it)|^2
 &\leq T^{1/3}
 \left|\frac{F_Y(1/2+\sigma+it)}
 {F_Y(1/2+\sigma)}\right|^2\times
 \prod_{Y<p\leq X}|1+f(p)p^{-1/2-\sigma-it}|^2 .
\end{align*}
The ratio uses only primes at most $Y$, whereas the last product uses only
primes exceeding $Y$.  Independence and
\eqref{eq:short-ratio}--\eqref{eq:long-product-second-moment} therefore give
\[
 \E\!\left[\mathbf 1_{\Omega_X^{(0)}}
 |F_X(1/2+\sigma+it)|^2\right]\ll T^{4/3}\log X.
\]
Since the dyadic interval has length $O(T)$ and its kernel is $O(T^{-2})$,
it follows
that
\[
 \E\left[\mathbf 1_{\Omega_X^{(0)}}\frac1{\log X}
 \int_{T\leq|t|\leq2T}
 \frac{|F_X(1/2+\sigma+it)|^2}{\sigma^2+t^2}\,\mathrm{d}t\right]
 \ll T^{1/3}.
\]

For $|t|\leq\sigma$, take $Y=X$. On
$\Omega_X^{(0)}$ we have
$|F_X(1/2+\sigma)|^2\leq\sigma^{1/3}$, while
\eqref{eq:short-ratio} remains valid. Since
\(
 \frac1{\log X}
 \int_{|t|\leq\sigma}\frac{\mathrm{d}t}{\sigma^2+t^2}
 \ll1,
\)
the same argument gives
\[
 \E\left[
 \mathbf 1_{\Omega_X^{(0)}}
 \frac1{\log X}
 \int_{|t|\leq\sigma}
 \frac{|F_X(1/2+\sigma+it)|^2}{\sigma^2+t^2}\,\mathrm{d}t
 \right]
 \ll\sigma^{1/3}.
\]
Summing over the dyadic values of $T$ therefore yields
\[
 \E\left[
 \mathbf 1_{\Omega_X^{(0)}}
 \frac1{\log X}
 \int_{|t|\leq T_0}
 \frac{|F_X(1/2+\sigma+it)|^2}{\sigma^2+t^2}\,\mathrm{d}t
 \right]
 \ll
 \sigma^{1/3}+\sum_{2^j\sigma\leq T_0}(2^j\sigma)^{1/3}
 \ll T_0^{1/3}.
\]
Markov's inequality now gives
\begin{equation}\label{eq:low-frequency-bound}
 \begin{aligned}
 &\p\left(
 \Omega_X^{(0)}\ \text{and}\
 \frac1{\log X}\int_{|t|\leq T_0}
 \frac{|F_X(1/2+\sigma+it)|^2}{\sigma^2+t^2}\,\mathrm{d}t
 >(\log\log X)^{-1/200}
 \right)\\
 &\qquad\ll
 (\log\log X)^{1/200}T_0^{1/3}
 \ll (\log\log X)^{-1/40}.
 \end{aligned}
\end{equation}

The other three ranges follow from \eqref{eq:mc3-old}.  Since $q=2/3$ and
$(a+b)^q\leq a^q+b^q$, decomposition into unit intervals gives
\begin{equation}\label{eq:middle-frequency-bound}
 \E\left(
 \frac1{\log X}\int_{T_0<|t|\leq1}
 \frac{|F_X(1/2+\sigma+it)|^2}{\sigma^2+t^2}\,\mathrm{d}t
 \right)^q
 \ll T_0^{-2q}(\log\log X)^{-q/2}=(\log\log X)^{-1/5}.
\end{equation}
For $1<|t|\leq(\log X)^{900}$, cover the range by intervals
$[N-1/2,N+1/2]$. The factor $(1+N^2)^{-1}$ from the kernel
similarly yields
\[
 \begin{aligned}
 &\E\left(
 \frac1{\log X}\int_{1<|t|\leq(\log X)^{900}}
 \frac{|F_X(1/2+\sigma+it)|^2}{\sigma^2+t^2}\,\mathrm{d}t
 \right)^q\\
 &\qquad\ll (\log\log X)^{-q/2}\sum_{N\in\Z}
 \frac{(\log\log(|N|+10))^q}{(1+N^2)^q}
 \ll (\log\log X)^{-1/3}.
 \end{aligned}
\]
The series converges because $2q=4/3>1$. Finally,
$\E|F_X(1/2+\sigma+it)|^2
=\prod_{p\leq X}(1+p^{-1-2\sigma})\ll\log X$, so the first moment of the
tail is
\[
 \E\left(
 \frac1{\log X}\int_{|t|>(\log X)^{900}}
 \frac{|F_X(1/2+\sigma+it)|^2}{\sigma^2+t^2}\,\mathrm{d}t
 \right)\ll
 \int_{|t|>(\log X)^{900}}\frac{\mathrm{d}t}{1+t^2}
 \ll(\log X)^{-900}.
\]

Applying Markov's inequality with threshold $(\log\log X)^{-1/200}$ multiplies the first two
$q$-th-moment estimates by $(\log\log X)^{q/200}=(\log\log X)^{1/300}$.  The resulting
probabilities, as well as the first-moment tail probability, are
$o((\log\log X)^{-1/200})$.  Combining these estimates with
\eqref{eq:good-real-axis} and \eqref{eq:low-frequency-bound} proves the
lemma.
\end{proof}

We can now state the estimate for the $X$-smooth contribution in the form
needed later.

\begin{proposition}[Mean square of the $X$-smooth contribution]
\label{prop:old-prime-energy}
\begin{equation}\label{eq:old-prime-energy}
 \p\left(\frac1{\log X}\int_{X^{8/7}}^{X^{4/3}}
 |C_X(x)|^2\frac{\mathrm{d}x}{x}>(\log\log X)^{1/3}\right)
 \ll (\log\log X)^{-1/200}.
\end{equation}
\end{proposition}

\begin{proof}
Set $\sigma=1/\log X$. Applying Parseval's identity
\eqref{eq:parseval-dirichlet} to the Dirichlet polynomial
$F_X(1/2+s)$ gives
\[
 \int_1^\infty |M_X(x)|^2x^{-1-2\sigma}\,\mathrm{d}x
 =\frac1{2\pi}\int_{-\infty}^{\infty}
 \frac{|F_X(1/2+\sigma+it)|^2}{\sigma^2+t^2}\,\mathrm{d}t.
\]
Since $x^{2\sigma}\leq\exp(8/3)$ for $x\leq X^{4/3}$,
Lemma~\ref{lem:weighted-squarefree-energy} and the preceding identity
imply, outside an event of probability $O((\log\log X)^{-1/200})$, that
\[
 \frac1{\log X}\int_{X^{8/7}}^{X^{4/3}}
 |M_X(x)|^2\frac{\mathrm{d}x}{x}
 \ll (\log\log X)^{-1/200}.
\]

On the other hand, Tonelli's theorem and the estimate
$\E|D_X(x)|^2\leq1$ give
\[
 \begin{aligned}
 \E\left(\frac1{\log X}\int_{X^{8/7}}^{X^{4/3}}
 |D_X(x)|^2\frac{\mathrm{d}x}{x}\right)
 &=
 \frac1{\log X}\int_{X^{8/7}}^{X^{4/3}}
 \E|D_X(x)|^2\frac{\mathrm{d}x}{x}\\
 &\leq
 \frac1{\log X}\log\left(X^{4/3-8/7}\right)
 =\frac4{21}.
 \end{aligned}
\]
Thus Markov's inequality shows that, outside an event of probability
$O((\log\log X)^{-1/3})$, the last integral is at most $(\log\log X)^{1/3}/4$. For sufficiently
large $X$, the corresponding integral for $M_X$ is also at most
$(\log\log X)^{1/3}/4$. Since
\(
 |C_X(x)|^2\leq2|M_X(x)|^2+2|D_X(x)|^2,
\)
the required bound follows by taking the union of the two exceptional
events and using $(\log\log X)^{-1/3}\leq (\log\log X)^{-1/200}$.
\end{proof}

An immediate consequence, useful when the endpoints are logarithmically
shifted, is that whenever the integral in \eqref{eq:old-prime-energy} is at
most $(\log\log X)^{1/3}$, one has, for every $A>0$,
\[
 \frac1{\log X}\int_{X^{8/7}}^{X^{4/3}}
 \mathbf 1_{\{|C_X(x)|>A\}}\frac{\mathrm{d}x}{x}
 \leq\frac{(\log\log X)^{1/3}}{A^2}.
\]

\section{The multiplier and conditional variance}\label{sec:variance}

\subsection{The  multiplier}

Recall $F_X$, $B_X$ and $Z_X$ from
\eqref{eq:def-squarefree-product}, \eqref{eq:def-B}
and \eqref{eq:def-zeta-product}.
The Dirichlet series of the extended Rademacher RMF restricted to $X$-smooth integers is 
\[
 \prod_{p\leq X}(1-f(p)p^{-s})^{-1}=F_X(s)Z_X(2s).
\]
This elementary factorisation is the point at which the extended
Rademacher RMF separates from the Rademacher RMF.

For $\sigma\geq 0$ and $t\in\R$, define
\begin{equation}\label{eq:def-multiplier}
 m_{X,\sigma}(t)
 \coloneqq
 (4^{\sigma+it}-1)Z_X(1+2\sigma+2it).
\end{equation}
The factor $4^{\sigma+it}-1$ removes the singular growth of $Z_X$ near
$\sigma=t=0$. We shall repeatedly use the following bounds.

\begin{lemma}\label{lem:multiplier}
For sufficiently large $X$, suppose that
\(
 0\leq \sigma\leq
 {100(\log\log X)^{1/100}}/{\log X}.
\)
Then the following hold.
\begin{enumerate}
\item If $|t|\leq1$, then $|m_{X,\sigma}(t)|\ll1$.
\item For every $t\in\R$,
\[
 |m_{X,\sigma}(t)|\ll
 \min\left(\log X,\,(1+\log(2+|t|))^2\right).
\]
\item If $1/3\leq |t|\leq1/2$, then
$|m_{X,\sigma}(t)|\asymp1$.
\end{enumerate}
All implied constants are absolute.
\end{lemma}

\begin{proof}
Expanding the logarithm of the Euler product, the terms coming from
prime powers $p^k$, $k\geq2$, contribute $O(1)$ uniformly in $\sigma$.
Hence
\begin{equation}\label{eq:ZX-prime-sum}
 \log |Z_X(1+2\sigma+2it)|
 =
 \Re\sum_{p\leq X}\frac{1}{p^{1+2\sigma+2it}}+O(1).
\end{equation}
Suppose first that $|t|\leq1$. By the prime number theorem and partial summation, we have
\[
\begin{aligned}
 \Re\sum_{p\leq X}\frac1{p^{1+2\sigma+2it}}
 &=
 \int_{\log 2}^{\log X}
 e^{-2\sigma v}\cos(2tv)\frac{\mathrm{d}v}{v}
 +O(1)\\
 &=
 \log\left(
 1+\min\left(\log X,\frac{1}{\sigma+|t|}\right)
 \right)+O(1).
 \end{aligned}
\]
To get the last line, we split the integral at \( \min\left(\log X,1/{(\sigma+|t|)}\right)
\)
and use exponential decay when $\sigma\geq|t|$ and integration by
parts when $|t|>\sigma$.
Since
\( |4^{\sigma+it}-1|\ll\sigma+|t|,
\)
we obtain
\[
 |m_{X,\sigma}(t)|
 \ll
 (\sigma+|t|)
 \left(
 1+\min\left(\log X,\frac{1}{\sigma+|t|}\right)
 \right)
 \ll1,
\]
which proves the first assertion.

For $|t|\geq1$, we have $|4^{\sigma+it}-1|\ll1$, while trivially
\(
 |Z_X(1+2\sigma+2it)|\ll\log X.
\)
If
\(
 \log X\ge C\log^2(2+|t|)
\) for some sufficiently large constant $C$,
 then by Mertens'
theorem and the prime number theorem, we have
\[
 \sum_{p\leq \exp(C\log^2(2+|t|))}\frac1{p^{1+2\sigma+2it}}
 \le 2\log\log(2+|t|)+O(1),
\]
and
\[
 \sum_{\exp(C\log^2(2+|t|))<p\leq X}
 \frac1{p^{1+2\sigma+2it}}
 \ll
 \frac1{|t|\log^2(2+|t|)}
 +(1+|t|)e^{-c\sqrt C\log(2+|t|)}
 \ll1.
\]
It follows from \eqref{eq:ZX-prime-sum} that
\[
 |Z_X(1+2\sigma+2it)|
 \ll(1+\log(2+|t|))^2.
\]
If instead \( \log X\le C\log^2(2+|t|),
\)
the same bound follows clearly from
\[
 |Z_X(1+2\sigma+2it)|
 \ll\log X\ll\log^2(2+|t|).
\]
Together with the trivial bound, this proves
\[
 |m_{X,\sigma}(t)|
 \ll
 \min\left(\log X,(1+\log(2+|t|))^2\right).
\]

For $1/3\leq|t|\leq1/2$, the prime number theorem and partial summation
give, uniformly in $\sigma$, 
\[ \sum_{p>X}\frac1{p^{1+2\sigma+2it}}\ll \frac{1}{|t|\log X}+e^{-c\sqrt{\log X}} =o(1),
\]
while the terms with prime powers $p^k$, $k\geq2$, contribute
$O(X^{-1})$.  Comparing the logarithms of the corresponding Euler
products, we obtain
\[
 Z_X(1+2\sigma+2it)
 =
 \zeta(1+2\sigma+2it)(1+o(1)).
\]
Since $\zeta(1+2\sigma+2it)$ and $4^{\sigma+it}-1$ are both bounded
above and away from zero uniformly on this range, it follows that
\(
 |m_{X,\sigma}(t)|\asymp1,
\) which proves the third assertion.
\end{proof}

For $\mathrm{Re}\,s>1$, by termwise integration we obtain the exact identity
\begin{equation}\label{eq:B-mellin}
 \int_0^\infty B_X(z)z^{-s-1}\,\mathrm{d}z
 =\left(4^{s-1/2}-1\right)\frac{F_X(s)Z_X(2s)}s
 \qquad(\mathrm{Re}\,s>1).
\end{equation}
For fixed $X$, the number of $X$-smooth integers at most $z$ is at most
\[
 \prod_{p\leq X}\left(1+\frac{\log z}{\log p}\right)
 \ll_X(1+\log z)^{\pi(X)}.
\]
Thus the Mellin transform on the left of \eqref{eq:B-mellin} converges
absolutely and is holomorphic for $\mathrm{Re}\,s>0$; there is no contribution near
zero because $B_X(z)=0$ for $z<1/4$. The right-hand side is also
holomorphic in this half-plane. Since both sides are holomorphic for $\Re s>0$ and agree for
$\Re s>1$, the identity theorem extends \eqref{eq:B-mellin}
to $\Re s>0$. In particular, on writing
$s=1/2+\sigma+it$, the
right-hand side of \eqref{eq:B-mellin} is
\[
 \frac{m_{X,\sigma}(t)F_X(1/2+\sigma+it)}
 {1/2+\sigma+it}.
\]
Thus the $\zeta$-type growth of $Z_X(2s)$ near $s=1/2$ is
cancelled uniformly in $X$ by the multiplier $4^{s-1/2}-1$.

\subsection{The conditional variance}

Conditioning on the variables $(f(p))_{p\leq X}$, the quantities
$B_X(x/p)$ are fixed, while the variables $f(p)$, $p>X$, remain
independent Rademacher random variables. Recall the definition of $Y_x$, we have
\[
\var_X(Y_x)=\mathrm{Var}\left(Y_x\mid(f(p))_{p\leq X}\right)
 =
 \frac1x\sum_{X<p\leq4x}B_X(x/p)^2.
\]
The following proposition shows that this conditional variance is
typically not too small. The prrof is an adaptation of the corresponding variance estimate
in \cite[Proposition 1]{Harper2013Gaussian}, with the  multiplier accounting for the
additional Euler factor arising in the extended Rademacher RMF.


\begin{proposition}[Conditional variance of the large-prime increment]
\label{prop:B-variance}
Let $W_0\leq W\leq(\log\log X)^{1/100}$, where $W_0$ is a
sufficiently large absolute constant, and suppose that
$X^{8/7}\leq x\leq X^{4/3}/4$. Then we have 
\begin{equation}\label{eq:B-variance-lower}
 \mathrm{Var}_X(Y_x)
 \gg \frac{e^{-2.2W}}{\sqrt{\log\log X}}
\end{equation}
with  probability $\ge1-O(e^{-0.1W})$.
\end{proposition}

\begin{proof}
 We first use the constancy of $B_X$ on
intervals of length $1/4$, together with Huxley's theorem, to bound the
prime sum from below by a weighted integral of $B_X^2$. We then apply
Parseval's identity \eqref{eq:parseval-dirichlet} and multiplicative-chaos estimates to obtain the
required lower bound.

For each integer $1\leq r\leq\lceil2X^{1/7}\rceil-1$, consider the
interval
\(
 4x/(r+1)<p\leq4x/r.
\)
Its left endpoint exceeds $X$ for all sufficiently large $X$, since
$x\geq X^{8/7}$, while its right endpoint is at most $4x$. Its length is
$\asymp x/r^2$ and its left endpoint is $\asymp x/r$. Moreover,
$r^8\ll x$, and hence
\(
 x/r^2\gg(x/r)^{6/7}.
\)
Since $6/7>7/12$, Huxley's theorem on primes in short intervals
\cite[Theorem~1]{Huxley1972} is available. Hence
\[
 \#\left\{p:\frac{4x}{r+1}<p\leq\frac{4x}{r}\right\}
 \gg\frac{x}{r(r+1)\log X}.
\]

For these primes, $x/p\in[r/4,(r+1)/4)$, so
$B_X(x/p)=B_X(r/4)$. Since
\(
 \int_{r/4}^{(r+1)/4}z^{-2}\,\mathrm{d}z=4/(r(r+1)),
\)
summing over $r$ we get
\[
 \frac1x\sum_{X<p\leq4x}B_X(x/p)^2
 \gg\frac1{\log X}
 \sum_{1\leq r\leq\lceil2X^{1/7}\rceil-1}
 \frac{B_X(r/4)^2}{r(r+1)}
 \gg\frac1{\log X}
 \int_{1/4}^{X^{1/7}/2}B_X(z)^2\frac{\mathrm{d}z}{z^2}.
\]

It remains to estimate the integral. Let $\eta=42W/\log X$. Since
$B_X(z)=0$ for $z<1/4$ and
$(X^{1/7}/2)^{-\eta}\leq2e^{-6W}$. By comparing $z^{-2}$ with $z^{-2-2\eta}$ on
$[1/4,X^{1/7}/2]$ and controlling the tail $z>X^{1/7}/2$ by
$z^{-2-\eta}$, we obtain
\[
\int_{1/4}^{X^{1/7}/2}B_X(z)^2\frac{\mathrm{d}z}{z^2}
 \geq4^{-2\eta}\left(
 \int_0^\infty B_X(z)^2\frac{\mathrm{d}z}{z^{2+2\eta}}
 -2e^{-6W}\int_0^\infty B_X(z)^2
 \frac{\mathrm{d}z}{z^{2+\eta}}
 \right).
\]
Parseval's identity \eqref{eq:parseval-dirichlet} and \eqref{eq:B-mellin} give
\begin{equation}\label{eq:B-main-parseval}
 \int_0^\infty B_X(z)^2\frac{\mathrm{d}z}{z^{2+2\eta}}
 =
 \frac1{2\pi}\int_{\R}
 \left|
 \frac{m_{X,\eta}(t)F_X(1/2+\eta+it)}
 {1/2+\eta+it}
 \right|^2\,\mathrm{d}t.
\end{equation}
The third assertion of Lemma~\ref{lem:multiplier}, followed by
Multiplicative Chaos Result~4 of
\cite[Section~2.4]{Harper2023Large}, shows that, outside an event of
probability $O(e^{-0.1W})$, this is
\[
 \gg\int_{1/3}^{1/2}|F_X(1/2+\eta+it)|^2\,\mathrm{d}t
 \gg\frac{\log X}
 {42W e^{2.1W}\sqrt{\log\log X}}.
\]

We next show that the second integral in the truncation error is
negligible. We claim that
\begin{equation}\label{eq:B-tail-moment}
 \E\left(
 \int_{\R}
 \left|
 \frac{m_{X,\eta/2}(t)F_X(1/2+\eta/2+it)}
 {1/2+\eta/2+it}
 \right|^2\,\mathrm{d}t
 \right)^{2/3}
 \ll
 \left(\frac{\log X}
 {W\sqrt{\log\log X}}\right)^{2/3}.
\end{equation}
For each integer $N$ with $|N|\leq\log^{1000}X$, apply
Multiplicative Chaos Result~3 of
\cite[Section~2.4]{Harper2023Large} on
$[N-1/2,N+1/2]$. Lemma~\ref{lem:multiplier} contributes at most
$O((1+\log(2+|N|))^4)$ before taking the $2/3$ power, and the resulting
bounds are summable because
\[
 \sum_{N\in\Z}
 \frac{(1+\log(2+|N|))^{8/3}
 (\log\log(|N|+10))^{2/3}}
 {(1+N^2)^{2/3}}
 <\infty.
\]
The subadditivity of $u^{2/3}$ allows us to sum the contributions of
these intervals. On the complementary range,
$|m_{X,\eta/2}(t)|\ll\log X$ and
$\E|F_X(1/2+\eta/2+it)|^2\ll\log X$, so
\[
 \E\int_{|t|>\log^{1000}X}
 \left|
 \frac{m_{X,\eta/2}(t)F_X(1/2+\eta/2+it)}
 {1/2+\eta/2+it}
 \right|^2\,\mathrm{d}t
 \ll
 \log^3X\int_{|t|>\log^{1000}X}\frac{\mathrm{d}t}{1+t^2}
 \ll\log^{-900}X.
\]
Together with concavity, this proves \eqref{eq:B-tail-moment}.

Parseval's identity \eqref{eq:parseval-dirichlet} with shift $\eta/2$ and Markov's inequality now show
that, outside an event of probability $O(e^{-2W})$,
\[
 \int_0^\infty B_X(z)^2\frac{\mathrm{d}z}{z^{2+\eta}}
 =
 \frac1{2\pi}\int_{\R}
 \left|
 \frac{m_{X,\eta/2}(t)F_X(1/2+\eta/2+it)}
 {1/2+\eta/2+it}
 \right|^2\,\mathrm{d}t
 \ll
 \frac{e^{3W}\log X}{W\sqrt{\log\log X}}.
\]
After multiplication by $2e^{-6W}$, this is smaller than the main term
once $W_0$ is sufficiently large. Since $4^{-2\eta}\asymp1$, we conclude,
outside an event of probability $O(e^{-0.1W})$, that
\[
 \var_X(Y_x)=\frac1x\sum_{X<p\leq4x}B_X(x/p)^2
 \gg
 \frac{e^{-2.1W}}{W\sqrt{\log\log X}}
 \gg
 \frac{e^{-2.2W}}{\sqrt{\log\log X}}.
\]
The probabilistic estimates depend only on $F_X$, so the same exceptional
event works simultaneously for all
$X^{8/7}\leq x\leq X^{4/3}/4$. This proves
\eqref{eq:B-variance-lower}.
\end{proof}

\section{Large conditional covariance estimates}\label{sec:covariances}

We next prove a weighted version of Harper's covariance estimate. For
$0\leq\theta\leq2\pi$ let
\[
 {\mathcal G}_X(\theta)\coloneqq\left\{
 X^{8/7}e^{\theta+2\pi r}:
 r\in\Z_{\geq0},\quad
 X^{8/7}e^{\theta+2\pi r}\leq X^{4/3}/4
 \right\}.
\]
For $x,y\in{\mathcal G}_X(\theta)$ we have
\[
 \cov_X(Y_x,Y_y)=\frac1{\sqrt{xy}}
 \sum_{X<p\leq X^{4/3}}B_X(x/p)B_X(y/p).
\]
Terms with $p>4x$ or $p>4y$ vanish. Throughout this section, $\theta$
is regarded as a fixed common shift. All estimates below are
uniform in $\theta$, and the exceptional events may be chosen independently
of $\theta$.

\begin{proposition}[Sparsity of large conditional covariances]
\label{prop:weighted-covariance}
Outside an event of probability $O((\log\log X)^{-0.1})$, we have,
simultaneously for every $0\leq\theta\leq2\pi$,
\[
 \max_{x\in{\mathcal G}_X(\theta)}
 \#\left\{y\in{\mathcal G}_X(\theta)\setminus\{x\}:
 |\cov_X(Y_x,Y_y)|>(\log\log X)^{-0.59}\right\}
 \leq(\log X)^{0.8}.
\]
\end{proposition}

We organise the proof in two stages. The first three lemmas form the
analytic reduction of the conditional covariance. We first use Perron's
formula to represent the covariance by a double integral, then
restrict this integral to the near-diagonal region, and finally insert the
barrier condition. The outcome is the approximation
\[
 \cov_X(Y_x,Y_y)
 =\frac{1}{4\pi^2}{\mathcal I}(x,y)
 +O((\log\log X)^{-0.6})
\]
outside a suitably small exceptional set.
The last two lemmas estimate ${\mathcal I}(x,y)$. 
The proposition then follows from the covariance representation under the
barrier condition and the uniform high moment bound by a counting argument.

Recall the definition of $m_{X,\sigma}(t)$ in Lemma \ref{lem:multiplier} and the definition of $ F_X(s)$ in \eqref{eq:def-squarefree-product}. For convenience, we also introduce the following shorthand throughout the proof.
\[
 T=X^{3/4},\quad K=(\log X)^{20},\quad
 d=\frac{(\log\log X)^{100}}{\log X},
 \quad
 P_X(u)\coloneqq\sum_{X<p\leq X^{4/3}}\frac1{p^{1+iu}}.
\]

\begin{lemma}[Perron formula and frequency truncation]
\label{lem:cov-perron}
Outside an event of probability $O(\log^{-1/2}X)$, we have,
simultaneously for every $0\leq\theta\leq2\pi$ and all
$x,y\in{\mathcal G}_X(\theta)$,
\begin{align}
 \cov_X(Y_x,Y_y)
 &=\frac1{4\pi^2}\int_{-K}^{K}\int_{-K}^{K}
 \frac{m_{X,0}(v)F_X(1/2+iv)x^{iv}}{1/2+iv}\notag\\
 &\qquad\times
 \frac{\overline{m_{X,0}(t)F_X(1/2+it)}y^{-it}}{1/2-it}
 P_X(v-t)\,\mathrm{d}v\,\mathrm{d}t+O(\log^{-3}X).
 \label{eq:compact-covariance}
\end{align}
\end{lemma}

\begin{proof}
Applying the truncated Perron formula
\cite[Theorem~5.2 and Corollary~5.3]{MontgomeryVaughan2007}
to the two partial sums in \eqref{eq:def-B}, we obtain
\[
 \frac{B_X(x/p)}{\sqrt x}
 =\frac1{2\pi i\sqrt x}
 \int_{1-iT}^{1+iT}
 \left(4^{s-1/2}-1\right)F_X(s)Z_X(2s)
 \frac{(x/p)^s}{s}\,\mathrm{d}s+E(x,p),
\]
where by the Corollary~5.3 of \cite{MontgomeryVaughan2007} with
$\sigma_0=1$, we get the error term is
\begin{align}
 E(x,p)\ll\frac1{\sqrt x}\sum_{a\in\{1,4\}}
 \Bigg(&
 \mathbf 1_{ax/p\in\mathbb Z}
 +\sum_{\substack{ax/(2p)<n<2ax/p\\ n\ne ax/p}}
 \min\bigg(1,\frac{ax}{T|ax-pn|}\bigg)
 +\frac{(1+ax/p)\log X}{T}
 \Bigg).
 \label{eq:perron-remainder}
\end{align}
The indicator accounts for the possible endpoint contribution.

Since the covariance is a quadratic sum over $p$, we do not have to need a 
pointwise estimate for $E(x,p)$. It is enough to obtain a power-saving
bound for its $\ell^2$-norm over the primes,
\begin{equation}\label{eq:perron-error-l2}
 \sum_{X<p\leq X^{4/3}}|E(x,p)|^2\ll X^{-c}
\end{equation}
for some absolute $c>0$. After the contour shift, we may combine this
with the corresponding $\ell^2$-bounds for the vertical and horizontal
contributions by Cauchy--Schwarz. We therefore estimate the
Perron remainders collectively over $p$. 

We fix $a\in\{1,4\}$ and make the change of variables $m=pn$
in the double sum above. Since
\(
 m<2ax\leq2X^{4/3}<X^2,
\)
each such $m$ has at most one prime divisor $p>X$. Hence
\begin{align*}
 &\sum_{X<p\leq X^{4/3}}
 \Bigg(
 \mathbf 1_{ax/p\in\mathbb Z}
 +\sum_{\substack{ax/(2p)<n<2ax/p\\ n\ne ax/p}}
 \min\bigg(1,\frac{ax}{T|ax-pn|}\bigg)
 \Bigg)\\
 &\qquad\qquad\ll
 1+\sum_{\substack{ax/2<m<2ax\\m\ne ax}}
 \min\bigg(1,\frac{ax}{T|ax-m|}\bigg)
 \ll 1+\frac{x\log X}{T}.
\end{align*}
The last inequality follows immediately by summing according to the distance
of $m$ from $ax$.
Using also the fact that
\(
 \#\{X<p\leq X^{4/3}\}\ll{X^{4/3}}/{\log X},
\) and \(
 \sum_{X<p\leq X^{4/3}}1/p\ll1,
\)
we obtain
\[
 \sum_{X<p\leq X^{4/3}}|E(x,p)|
 \ll
 \frac1{\sqrt x}
 \left(1+\frac{x\log X}{T}+\frac{X^{4/3}}{T}\right)
 \ll X^{1/20}.
\]
On the other hand, from the same harmonic-sum estimate applied pointwise,
together with $ax/p\ll X^{1/3}$, we obtain
\(
 \max_{X<p\leq X^{4/3}}|E(x,p)|
 \ll x^{-1/2}\ll X^{-1/2}.
\)
Thus \eqref{eq:perron-error-l2} follows.


We now move the contour to $\operatorname{Re}s=1/2$. No pole is crossed.
We first show that the two horizontal segments make a negligible
contribution. Since $Z_X$ is deterministic, by Cauchy--Schwarz and
independence we obtain, uniformly for $1/2\leq\sigma\leq1$,
\[
 \begin{aligned}
 \E|F_X(\sigma+iT)Z_X(2\sigma+2iT)|
 &\leq |Z_X(2\sigma+2iT)|
 \sqrt{\E|F_X(\sigma+iT)|^2}\\
 &\ll
 \prod_{p\leq X}(1-p^{-2\sigma})^{-1}
 \prod_{p\leq X}(1+p^{-2\sigma})^{1/2}
 \ll \log^{3/2}X.
 \end{aligned}
\]
Hence, by Markov's inequality,
\[
 \int_{1/2}^1
 |F_X(\sigma+iT)Z_X(2\sigma+2iT)|\,\mathrm{d}\sigma
 \leq\log^2X
\]
outside an event of probability $O(\log^{-1/2}X)$; the same conclusion
holds for the lower horizontal segment.
On this event, the absolute value of either horizontal contribution
corresponding to $(x,p)$ is
\[
 \begin{aligned}
 &\ll \frac1{\sqrt x}
 \int_{1/2}^1
 |F_X(\sigma\pm iT)Z_X(2\sigma\pm2iT)|
 \frac{(x/p)^\sigma}{|\sigma\pm iT|}
 \,\mathrm{d}\sigma\ll \frac{\log^2X}{T}
 \left(\frac1{\sqrt p}+\frac{\sqrt x}{p}\right),
 \end{aligned}
\]
where we used $|\sigma\pm iT|\asymp T$ and
\(
 \left({x}/{p}\right)^\sigma
 \ll \sqrt{x/p}+x/p \) if \(\frac12\leq\sigma\leq1.
\)
Consequently,
\[
 \begin{aligned}
 \sum_{X<p\leq X^{4/3}}
 \left[
 \frac{\log^2X}{T}
 \left(\frac1{\sqrt p}+\frac{\sqrt x}{p}\right)
 \right]^2
 &\ll
 \frac{\log^4X}{T^2}
 \left(1+\frac{x}{X}+\sqrt{\frac{x}{X}}\right)\ll X^{-c}
 \end{aligned}
\]
for some absolute $c>0$, since $T=X^{3/4}$ and
$x\leq X^{4/3}/4$. Thus the horizontal contributions have
power-saving $\ell^2$-norm over the primes and may later be absorbed by
Cauchy--Schwarz in the covariance calculation.

On the vertical segment, writing $s=1/2+iv$, the contribution to
$B_X(x/p)/\sqrt x$ is
\[
 \frac1{2\pi\sqrt p}
 \int_{-T}^{T}
 \frac{m_{X,0}(v)F_X(1/2+iv)x^{iv}p^{-iv}}
 {1/2+iv}\,\mathrm{d}v.
\]
Moreover, from the trivial multiplier bound, Cauchy--Schwarz, and
$\E|F_X(1/2+iv)|^2\ll\log X$, we obtain
\[
 \E\int_{-T}^{T}
 \frac{|m_{X,0}(v)F_X(1/2+iv)|}{1+|v|}\,\mathrm{d}v
 \ll\log^{5/2}X.
\]
Thus, outside an event of probability $O(\log^{-3/2}X)$, the last integral
is at most $\log^4X$. It follows that the squared sum over $p$ of the
vertical contributions is $O(\log^8X)$.

Combining this bound with \eqref{eq:perron-error-l2} and the preceding
bound for the horizontal contributions, Cauchy--Schwarz shows that every
covariance term containing either a Perron remainder or a horizontal
contribution is $O(X^{-c})$ for some $c>0$.
Consequently, outside an event of probability $O(\log^{-1/2}X)$, the
covariance is given, up to $O(X^{-c})$, by the product of the two vertical
integrals.

Since
\(
 F_X(1/2-it)=\overline{F_X(1/2+it)}
\)
and
\(
 m_{X,0}(-t)=\overline{m_{X,0}(t)},
\)
by changing $t$ to $-t$ in the second vertical integral and summing over
$p$, we obtain
\begin{align}
 \cov_X(Y_x,Y_y)
 &=\frac1{4\pi^2}
 \int_{-T}^{T}\int_{-T}^{T}
 \frac{m_{X,0}(v)F_X(1/2+iv)x^{iv}}
 {1/2+iv}\notag\\
 &\qquad\times
 \frac{\overline{m_{X,0}(t)F_X(1/2+it)}y^{-it}}
 {1/2-it}
 P_X(v-t)\,\mathrm{d}v\,\mathrm{d}t
 +O(X^{-c}).
 \label{eq:combined-perron}
\end{align}

It remains to truncate the frequencies. Applying Number Theory Result~1 of
\cite[Section~2.2]{Harper2023Large} with
$a_p=p^{it}/\log p$ for $X<p\leq X^{4/3}$ and $a_n=0$
otherwise, we obtain uniformly in $t$,
\[
 \int_{-T}^{T}|P_X(v-t)|^2\,\mathrm{d}v
 \ll
 \sum_{X<p\leq X^{4/3}}\frac1{p\log p}
 \ll\frac1{\log X}.
\]
Here the condition $p\geq T^{1.01}$ is satisfied since
$T=X^{3/4}$. Moreover, by Cauchy--Schwarz and independence, we have \(
 \E|F_X(1/2+iv)F_X(1/2+it)|\ll\log X,
\)
while Lemma~\ref{lem:multiplier} gives
\(
 |m_{X,0}(v)m_{X,0}(t)|\ll\log^2X.
\)

By symmetry, it is enough to consider the part of
\eqref{eq:combined-perron} where \(|v|\geq\max(K,|t|).
\)
Using the preceding bounds for the Euler products and the multiplier, the
expected absolute value of this contribution is
\[
 \ll \log^3X
 \int_{|t|\leq T}\frac{\mathrm{d}t}{1+|t|}
 \int_{\substack{\max(K,|t|)\leq|v|\leq T}}
 \frac{|P_X(v-t)|}{1+|v|}\,\mathrm{d}v.
\]
For fixed $t$, Cauchy--Schwarz and the mean value estimate for $P_X$ give
\[
 \begin{aligned}
 \int_{\substack{\max(K,|t|)\leq|v|\leq T}}
 \frac{|P_X(v-t)|}{1+|v|}\,\mathrm{d}v
 &\leq
 \left(\int_{-T}^{T}|P_X(v-t)|^2\,\mathrm{d}v\right)^{1/2}
 \left(\int_{|v|\geq\max(K,|t|)}
 \frac{\mathrm{d}v}{(1+|v|)^2}\right)^{1/2}\\
 &\ll
 \frac{1}{\sqrt{\log X}\sqrt{\max(K,|t|)}}.
 \end{aligned}
\]
Consequently the expected high frequency contribution is
\[
 \ll
 \log^{5/2}X
 \int_{\mathbb R}
 \frac{\mathrm{d}t}
 {(1+|t|)\sqrt{\max(K,|t|)}}
 \ll
 \frac{\log^{5/2}X\log K}{\sqrt K}.
\]
Here the last estimate follows by splitting the $t$-integral at $|t|=K$.
Since $K=(\log X)^{20}$, this is $O(\log^{-7}X)$. Markov's inequality
therefore shows that the range $\max(|v|,|t|)\geq K$ contributes
$O(\log^{-3}X)$ outside an event of probability $O(\log^{-4}X)$.

Combining this with \eqref{eq:combined-perron} proves
\eqref{eq:compact-covariance}, and hence the lemma.
\end{proof}

The next lemma isolates the narrow near-diagonal frequency range on which
the barrier argument will be applied.

\begin{lemma}[Near-diagonal reduction]
\label{lem:cov-near-diagonal}
Let
\[
 \Delta=\{(v,t):(\log\log X)^{-2}\leq |v|,|t|\leq(\log\log X)^2,
 \ |v-t|\leq d\}.
\]
Outside an event of probability $O((\log\log X)^{-1.2})$, uniformly
for every $0\leq\theta\leq2\pi$ and all
$x,y\in{\mathcal G}_X(\theta)$, the contribution  from
\(
 [-K,K]^2\setminus\Delta
\) to the double integral in
\eqref{eq:compact-covariance} is $O((\log\log X)^{-0.7})$.
\end{lemma}

\begin{proof}
We first discard the off-diagonal range $|v-t|>d$. After removing the
finitely many primes below $200$, Euler Product Result~2 of
\cite[Section~2.3]{Harper2023Large} shows that, uniformly for
$|v|,|t|\leq K$,
\begin{align}
 \E\big|F_X(1/2+iv)F_X(1/2+it)\big|
 &\ll \sqrt{\log X}\bigg(
 1+\min\left(\sqrt{\log X},|v-t|^{-1/2}\right)\notag\\
 &+\min\left(\sqrt{\log X},|v+t|^{-1/2}\right)
 +\log(2+|v|+|t|)\bigg).
 \label{eq:two-point-R-far}
\end{align}
At $v=t$ or $v=-t$, the second entry in the corresponding minimum is
interpreted as $+\infty$. 

On the other hand, Number Theory Result~2 of
\cite[Section~2.2]{Harper2023Large}, combined with the trivial bound
$|P_X(u)|\ll1$, gives
\begin{equation}\label{eq:P-pointwise-far}
 |P_X(u)|\ll
 \min\left(1,\frac1{|u|\log X}\right)
 +O\left((1+K)e^{-c\sqrt{\log X}}\right)
 \qquad (|u|\leq2K).
\end{equation}
Since $d>1/\log X$, on the range $|v-t|>d$ we have
\(
 |P_X(v-t)|\ll ({|v-t|\log X})^{-1},
\)
up to the exponentially small error.
Lemma~\ref{lem:multiplier}, \eqref{eq:two-point-R-far}, and
\eqref{eq:P-pointwise-far} show that the expected absolute value of the
part with $|v-t|>d$ is
\begin{align}
 \ll \frac{(\log\log X)^4}{\sqrt{\log X}}
 &\iint_{\substack{|v|,|t|\leq K\\ |v-t|>d}}
 \frac{1}{(1+|v|)(1+|t|)|v-t|}
 \bigg(
 1+\min\left(\sqrt{\log X},|v-t|^{-1/2}\right)\notag\\
 &\qquad\ \ +\min\left(\sqrt{\log X},|v+t|^{-1/2}\right)
 +\log(2+|v|+|t|)\bigg)
 \,\mathrm{d}v\,\mathrm{d}t .
 \label{eq:far-expectation-expanded}
\end{align}
Here $(\log\log X)^4$ comes from the two multiplier bounds
$|m_{X,0}(v)|,|m_{X,0}(t)|\ll(\log\log X)^2$ on this range.
To estimate the regular terms, change variables $u=v-t$ and use
\[
 \int_{\mathbb R}
 \frac{\mathrm{d}r}{(1+|r|)(1+|r-u|)}
 \ll\frac{\log(2+|u|)}{1+|u|}.
\]
It follows that the terms
$1+\log(2+|v|+|t|)$ in
\eqref{eq:far-expectation-expanded} contribute $O(\log\log X)$ to the double integral, while
the singularity at $v-t=0$ contributes
\begin{equation}\label{eq:difference-singularity}
 \ll \int_d^1\frac{\mathrm{d}u}{u^{3/2}}+\log\log X
 \ll d^{-1/2}+\log\log X.
\end{equation}
For the term singular at $v+t=0$, we make the change
of variables \( u=v-t \) and \(s=v+t.
\)
Up to absolute constants arising from the Jacobian and the rescaling of
the two denominator factors, the relevant contribution is bounded by
\begin{align}
 \int_{d<|u|\leq2K}\frac{\mathrm{d}u}{|u|}
 \int_{|s|\leq2K}
 \frac{\min\left(\sqrt{\log X},|s|^{-1/2}\right)}
 {(1+|s+u|)(1+|s-u|)}\,\mathrm{d}s
 \ll\log(2/d).
 \label{eq:rad-sum-kernel}
\end{align}
The inner integral is $O(1)$ for $|u|\leq2$, while for
$|u|>2$ it is
$O((1+\log|u|)|u|^{-3/2})$ after splitting near $s=\pm u$ and on the
complement. 

Combining the three estimates, we have
\[
\eqref{eq:far-expectation-expanded} \ll
 \frac{(\log\log X)^4}{\sqrt{\log X}}
 \left(d^{-1/2}+\log\log X+\log(2/d)\right)
 \ll(\log\log X)^{-46}.
\]
Markov's inequality
therefore shows that the region $|v-t|>d$ may be discarded with error
$O((\log\log X)^{-20})$ outside an event of probability
$O((\log\log X)^{-26})$.

We are now left with $|v|,|t|\leq K$ and $|v-t|\leq d$.  We next
discard the ranges where
$\max(|v|,|t|)\geq(\log\log X)^2$ or
$\min(|v|,|t|)\leq(\log\log X)^{-2}$, leaving precisely the region
$\Delta$.
We shall use the $L^1$-norm of $P_X(u)$ 
on the near-diagonal
range.  From \eqref{eq:P-pointwise-far},
\begin{align}
 \int_{|u|\leq d}|P_X(u)|\,\mathrm{d}u
 &\ll
 \int_{|u|\leq1/\log X}\mathrm{d}u
 +\frac1{\log X}
 \int_{1/\log X<|u|\leq d}\frac{\mathrm{d}u}{|u|}\notag\\
 &\ll
 \frac{1+\log(d\log X)}{\log X}
 \ll\frac{\log\log\log X}{\log X}.
 \label{eq:near-kernel-mass}
\end{align}
Moreover, Cauchy--Schwarz in probability gives, uniformly in $v$ and $t$,
\[
 \E|F_X(1/2+iv)F_X(1/2+it)|\ll\log X.
\]
Thus the factor $\log X$ from the Euler product moment is cancelled by
the factor $1/\log X$ in \eqref{eq:near-kernel-mass}.

Consider first the range
\(
 \max(|v|,|t|)\geq(\log\log X)^2.
\)
Since $d=o((\log\log X)^2)$, one has $|v|\asymp|t|$ there.  Using
Lemma~\ref{lem:multiplier} and integrating first in $u=v-t$, the expected
absolute value of the contribution of this range to
\eqref{eq:compact-covariance} is
\[
 \ll
 (\log\log\log X)
 \int_{|r|\geq(\log\log X)^2/2}
 \frac{(1+\log(2+|r|))^4}{(1+|r|)^2}\,\mathrm{d}r
 \ll
 \frac{(\log\log\log X)^5}{(\log\log X)^2}.
\]
Here the factor $\log\log\log X$ comes from
\eqref{eq:near-kernel-mass}, while the remaining integral accounts for
the two multiplier factors and the two Mellin denominators.

Next consider
\( \min(|v|,|t|)\leq(\log\log X)^{-2}.
\)
Since $d=o((\log\log X)^{-2})$, both $v$ and $t$ are
$O((\log\log X)^{-2})$.  Lemma~\ref{lem:multiplier} then gives bounded
multipliers, and the remaining frequency interval has length
$O((\log\log X)^{-2})$.  Using \eqref{eq:near-kernel-mass} again, the
expected absolute value of this contribution is therefore
\(
 \ll
 {\log\log\log X}/{(\log\log X)^2}.
\)
Thus the total expected contribution from these two frequency ranges is
\(
 \ll
 {(\log\log\log X)^5}/{(\log\log X)^2}.
\)
Markov's inequality shows that they may be discarded with total error
$O((\log\log X)^{-0.7})$ outside an event of probability
$O((\log\log X)^{-1.2})$.

Together with the preceding removal of the region $|v-t|>d$, this
restricts the integral in \eqref{eq:compact-covariance} to $\Delta$ with
the claimed error and exceptional probability.  After taking absolute
values, the phase factors $x^{iv}$ and $y^{-it}$ disappear, so all the
majorants above are independent of $x$, $y$, and the common shift
$\theta$.  Hence the same exceptional event works uniformly for all the
grids under consideration.  
\end{proof}

We now insert the same strong barrier condition as in
\cite[Section~3.3]{Harper2023Large}.  For integers $0\leq j\leq\lfloor\log\log X\rfloor$, write
\[
 F_{X,j}(s)\coloneqq\prod_{p\leq X^{e^{-j}}}(1+f(p)p^{-s}).
\]
Following \cite[Section~2.4]{Harper2023Large}, for each $t\in\mathbb R$
we define a sequence of discretisations $t(j)$ by setting $t(-1)=t$ and,
for $0\leq j\leq\lfloor\log\log X\rfloor-1$,
\[
 t(j)\coloneqq
 \max\left\{
 u\leq t(j-1):
 u=\frac{n}{
 ((\log X)/e^j)\log((\log X)/e^j)}
 \ \text{for some }n\in\mathbb Z
 \right\}.
\]

For $t\in\mathbb R$, let ${\mathcal A}(t)$ be the event that, for every
integer $0\leq j\leq\lfloor\log\log X\rfloor-1$,
\begin{equation}\label{eq:strong-barrier}
 |F_{X,j}(1/2+it)|
 \leq
 \begin{cases}
 \displaystyle
 \frac{\log X}{e^j(\log\log X)^{1000}},
 & 0\leq j\leq\lfloor0.99\log\log X\rfloor,\\[2mm]
 \displaystyle
 \frac{\log X}{e^j}(\log\log X)^6,
 & \lfloor0.99\log\log X\rfloor<j
 \leq\lfloor\log\log X\rfloor-1.
 \end{cases}
\end{equation}
Define
\begin{equation}\label{eq:weighted-main-integral}
 {\mathcal I}(x,y)\coloneqq\iint_\Delta
 \frac{\mathbf 1_{{\mathcal A}(v)}m_{X,0}(v)F_X(1/2+iv)x^{iv}}{1/2+iv}
 \frac{\mathbf 1_{{\mathcal A}(t)}
 \overline{m_{X,0}(t)F_X(1/2+it)}y^{-it}}{1/2-it}
 P_X(v-t)\,\mathrm{d}v\,\mathrm{d}t.
\end{equation}

\begin{lemma}[Inserting the barrier condition]
\label{lem:cov-barrier}
Outside an event of probability $O((\log\log X)^{-0.1})$, we have,
simultaneously for every $0\leq\theta\leq2\pi$ and all
$x,y\in{\mathcal G}_X(\theta)$,
\[
 \cov_X(Y_x,Y_y)=\frac1{4\pi^2}{\mathcal I}(x,y)
 +O((\log\log X)^{-0.6}).
\]
\end{lemma}
\begin{proof}
We first introduce the weaker discretised barrier needed to apply the
multiplicative-chaos estimates.  For $t\in\mathbb R$, let
${\mathcal D}(t)$ be the event
\[
 |F_{X,j}(1/2+it(j))|
 \leq(\log X)e^{-j}(\log\log X)^5
 \qquad
 (0\leq j\leq\lfloor\log\log X\rfloor-1),
\]
where $t(j)$ is the discretisation defined above.  Let ${\mathcal D}$ be
the event that ${\mathcal D}(t)$ holds for every
$|t|\leq(\log\log X)^2$.

Applying Multiplicative Chaos Result~1 of
\cite[Section~2.4]{Harper2023Large} with the choice
\(
 e^W=(\log\log X)^3,
\)
the upper barrier is at most
$(\log X)e^{-j}(\log\log X)^5$, and the exceptional probability is
$O((\log\log X)^{-6})$.  Hence, for any fixed $N\in\mathbb R$,
${\mathcal D}(t)$ holds for every $|t-N|\leq1/2$ outside such an
exceptional event.  A union bound over
$O((\log\log X)^2)$ intervals of length 1 covering
$[-(\log\log X)^2,(\log\log X)^2]$ therefore yields
\[
 \p({\mathcal D}^c)\ll(\log\log X)^{-4}.
\]
On ${\mathcal D}$, the events ${\mathcal D}(v)$ and ${\mathcal D}(t)$
hold throughout $\Delta$, so their indicators may be inserted without
any error.

We now pass to the strong pointwise barrier ${\mathcal A}(t)$.  The range
\(
 (\log\log X)^{-2}\leq|t|\leq(\log\log X)^2
\)
lies within the  range of Multiplicative Chaos Result~2 of
\cite[Section~2.4]{Harper2023Large}.  Hence, uniformly in this range,
\begin{equation}\label{eq:pointwise-A-failure}
 \E\left(
 |F_X(1/2+it)|^2
 \mathbf 1_{{\mathcal D}(t)}
 \mathbf 1_{{\mathcal A}(t)^c}
 \right)
 \ll
 \frac{\log X(\log\log\log X)^7}{\log\log X}.
\end{equation}
Integrating against $|1/2+it|^{-2}$ and applying Markov's inequality,
we obtain, outside an event of probability
$O((\log\log X)^{-0.1})$,
\begin{equation}\label{eq:A-failure-energy}
 \int_{(\log\log X)^{-2}\leq|t|\leq(\log\log X)^2}
 \frac{
 \mathbf 1_{{\mathcal D}(t)}
 \mathbf 1_{{\mathcal A}(t)^c}
 |F_X(1/2+it)|^2}
 {|1/2+it|^2}\,\mathrm{d}t
 \ll
 \frac{\log X(\log\log\log X)^7}
 {(\log\log X)^{0.9}}.
\end{equation}

We also need an unrestricted mean square bound for the Euler product.
Take $q=9/10$ in Multiplicative Chaos Result~3 of
\cite[Section~2.4]{Harper2023Large}.  Decomposing the $t$-range into
intervals $[N-1/2,N+1/2]$ and using $(a+b)^q\leq a^q+b^q$, we obtain
\[
\begin{aligned}
 &\E\left(
 \int_{-(\log\log X)^2}^{(\log\log X)^2}
 \frac{|F_X(1/2+it)|^2}{|1/2+it|^2}\,\mathrm{d}t
 \right)^q\\
 &\qquad\ll
 \sum_{|N|\leq(\log\log X)^2+1}
 \frac1{(1+N^2)^q}
 \E\left(
 \int_{N-1/2}^{N+1/2}
 |F_X(1/2+it)|^2\,\mathrm{d}t
 \right)^q\\
 &\qquad\ll
 \left(\frac{\log X}{\sqrt{\log\log X}}\right)^q
 \sum_{N\in\mathbb Z}
 \frac{(\log\log(|N|+10))^q}{(1+N^2)^q}\ll
 \left(\frac{\log X}{\sqrt{\log\log X}}\right)^q.
\end{aligned}
\]
Here the last series converges since $2q=9/5>1$.
Therefore, by Markov's inequality,
\[
\begin{aligned}
 &\p\left(
 \int_{-(\log\log X)^2}^{(\log\log X)^2}
 \frac{|F_X(1/2+it)|^2}{|1/2+it|^2}\,\mathrm{d}t
 >
 \frac{\log X}{(\log\log X)^{0.35}}
 \right)\\
 &\qquad\ll
 \left(\frac{\log X}{\sqrt{\log\log X}}\right)^q
 \left(\frac{(\log\log X)^{0.35}}{\log X}\right)^q
 =
 (\log\log X)^{-0.135}.
\end{aligned}
\]
Thus, outside an event of probability
$O((\log\log X)^{-0.135})$,
\begin{equation}\label{eq:unrestricted-energy}
 \int_{-(\log\log X)^2}^{(\log\log X)^2}
 \frac{|F_X(1/2+it)|^2}{|1/2+it|^2}\,\mathrm{d}t
 \ll
 \frac{\log X}{(\log\log X)^{0.35}}.
\end{equation}

It remains to estimate the change in the near-diagonal integral caused by
inserting the strong barrier.  By \eqref{eq:near-kernel-mass},
\[
 \sup_v\int_{\{t:(v,t)\in\Delta\}}|P_X(v-t)|\,\mathrm{d}t
 +
 \sup_t\int_{\{v:(v,t)\in\Delta\}}|P_X(v-t)|\,\mathrm{d}v
 \ll
 \frac{\log\log\log X}{\log X}.
\]
Then for any square-integrable functions $F$ and $G$,
Cauchy--Schwarz gives us
 \begin{align}
 &\iint_{\Delta}
 |P_X(v-t)|\,|F(v)G(t)|\,\mathrm{d}v\,\mathrm{d}t\notag\\
 &\qquad\leq
 \left(
 \iint_{\Delta}
 |P_X(v-t)|\,|F(v)|^2\,\mathrm{d}v\,\mathrm{d}t
 \right)^{1/2}
 \left(
 \iint_{\Delta}
 |P_X(v-t)|\,|G(t)|^2\,\mathrm{d}v\,\mathrm{d}t
 \right)^{1/2}\notag\\
 &\qquad\ll
 \frac{\log\log\log X}{\log X}
 \|F\|_2\|G\|_2.\label{eq:CSinq}
 \end{align}

For all $v,t\in\mathbb R$, we have the elementary inequality
\[
 1-\mathbf 1_{{\mathcal A}(v)}
      \mathbf 1_{{\mathcal A}(t)}
 \leq
 \mathbf 1_{{\mathcal A}(v)^c}
 +
 \mathbf 1_{{\mathcal A}(t)^c}.
\]
On the event ${\mathcal D}$, we have throughout $\Delta$, \(
 \mathbf 1_{{\mathcal A}(v)^c}
 =
 \mathbf 1_{{\mathcal D}(v)}
 \mathbf 1_{{\mathcal A}(v)^c},\) \( \mathbf 1_{{\mathcal A}(t)^c}
 =
 \mathbf 1_{{\mathcal D}(t)}
 \mathbf 1_{{\mathcal A}(t)^c}.
\)
Hence, after taking absolute values, the difference between near-diagonal integral (i.e., the double
integral in \eqref{eq:compact-covariance} restricted to $\Delta$) and
${\mathcal I}(x,y)$ is bounded by the sum of two terms, the first of
which is
\[
 \begin{aligned}
 \iint_{\Delta}
 &|P_X(v-t)|\,|m_{X,0}(v)m_{X,0}(t)|
 \frac{
 \mathbf 1_{{\mathcal D}(v)}
 \mathbf 1_{{\mathcal A}(v)^c}
 |F_X(1/2+iv)|
 }{|1/2+iv|}
 \frac{|F_X(1/2+it)|}{|1/2-it|}
 \,\mathrm{d}v\,\mathrm{d}t,
 \end{aligned}
\]
By Lemma~\ref{lem:multiplier}, throughout $\Delta$, we have
\(
 |m_{X,0}(v)|,\ |m_{X,0}(t)|
 \ll(\log\log\log X)^2.
\)
Together with \eqref{eq:CSinq}, the displayed integral above is therefore
\[
 \begin{aligned}
 &\ll
 \frac{(\log\log\log X)^5}{\log X}
 \left(
 \int_{(\log\log X)^{-2}\leq|v|\leq(\log\log X)^2}
 \frac{
 \mathbf 1_{{\mathcal D}(v)}
 \mathbf 1_{{\mathcal A}(v)^c}
 |F_X(1/2+iv)|^2}
 {|1/2+iv|^2}\,\mathrm{d}v
 \right)^{1/2}\\
 &\qquad\qquad\times
 \left(
 \int_{-(\log\log X)^2}^{(\log\log X)^2}
 \frac{|F_X(1/2+it)|^2}
 {|1/2+it|^2}\,\mathrm{d}t
 \right)^{1/2}\\
 &\ll
 \frac{(\log\log\log X)^5}{\log X}
 \left(
 \frac{\log X(\log\log\log X)^7}
      {(\log\log X)^{0.9}}
 \right)^{1/2}
 \left(
 \frac{\log X}
      {(\log\log X)^{0.35}}
 \right)^{1/2}\ll
 (\log\log X)^{-0.6}.
 \end{aligned}
\]
Here in the last  line we have used \eqref{eq:A-failure-energy} and
\eqref{eq:unrestricted-energy}. The same bound holds for the second term in which ${\mathcal A}(t)$ fails.

Thus, outside an event of probability
$O((\log\log X)^{-0.1})$, inserting
$\mathbf 1_{{\mathcal A}(v)}\mathbf 1_{{\mathcal A}(t)}$ into the
near-diagonal integral changes it by
$O((\log\log X)^{-0.6})$.

Combining this with Lemmas~\ref{lem:cov-perron} and
\ref{lem:cov-near-diagonal}, we obtain
\[
 \cov_X(Y_x,Y_y)
 =
 \frac1{4\pi^2}{\mathcal I}(x,y)
 +O((\log\log X)^{-0.6}).
\]
The union of the exceptional events arising above and in
Lemmas~\ref{lem:cov-perron} and \ref{lem:cov-near-diagonal}
has probability $O((\log\log X)^{-0.1})$.

Finally, all these exceptional events depend only on the variables
$(f(p))_{p\leq X}$.  The dependence on $x$ and $y$ in the integral occurs
only through the factors $x^{iv}$ and $y^{-it}$, both of modulus one, so
all the estimates are uniform in $x$, $y$, and the common shift
$\theta$.  Hence the same exceptional event works simultaneously for
every $0\leq\theta\leq2\pi$ and all
$x,y\in{\mathcal G}_X(\theta)$.
\end{proof}

We next adapt Harper's high moment argument
\cite[Proposition~3 and Section~3.4]{Harper2023Large} to the weighted
integral ${\mathcal I}(x,y)$.  We retain the local mean square integrals
of the Euler product at this stage, while keeping track of the additional
deterministic multiplier.

\begin{lemma}[Weighted high moment estimate]
\label{lem:weighted-high-moment}
For any large $X$ and every integer $1\leq k\leq\log^{1/4}X$, outside an event of
probability at most $(\log\log X)^{-2k}$, we have, simultaneously for
every $0\leq\theta\leq2\pi$,
\begin{align}
 &\max_{x\in{\mathcal G}_X(\theta)}
 \sum_{y\in{\mathcal G}_X(\theta)}|{\mathcal I}(x,y)|^{2k}\notag\\
 &\ll \log X\bigg(
 \frac1{\log^{1/3}X}
 \sum_{|N|\leq(\log\log X)^2+1}\frac1{1+N^2}
 \left(
 \frac{C(\log\log\log X)^5}{\log X}
 \int_N^{N+1}|F_X(1/2+it)|^2\,\mathrm{d}t
 \right)^{2k}\notag\\
 &\hspace{32mm}
 +\frac1{\log^{1/3}X}
 \left(k^{29}(\log\log X)^{17}\right)^{2k}
 +\left(
 \frac{k^{29}}{(\log\log X)^{989}}
 \right)^{2k}
 \bigg). \label{eq:weighted-2k}
\end{align}
\end{lemma}

\begin{proof}
We follow the proof of
\cite[Proposition~3 and Section~3.4]{Harper2023Large}, keeping track of
the two modifications arising in the present setting.  The common shift
$\theta$ contributes only a phase of modulus one.  In addition, each copy
of ${\mathcal I}(x,y)$ contains the two deterministic factors
$m_{X,0}(v)$ and $m_{X,0}(t)$.  On $\Delta$, by Lemma~\ref{lem:multiplier} we have
\(
 |m_{X,0}(v)|,\ |m_{X,0}(t)|
 \ll(\log\log\log X)^2,
\)
so the $2k$-th moment expansion contains $4k$ multiplier factors and
hence incurs the loss
\begin{equation}\label{eq:multiplier-2k-loss}
 (C\log\log\log X)^{8k}.
\end{equation}

For each integer $N$, let $\Delta_N$ be the part of $\Delta$ on which \( N\leq\max(v,t)<N+1,
\)
and let ${\mathcal I}_N(x,y)$ denote the corresponding part of
${\mathcal I}(x,y)$.  By H\"older's inequality we get
\begin{equation}\label{eq:unit-holder}
 |{\mathcal I}(x,y)|^{2k}
 \ll C^{2k}\sum_N\frac1{1+N^2}
 \left|(1+N^2){\mathcal I}_N(x,y)\right|^{2k}.
\end{equation}
On $\Delta_N$, we have \( |(1/2+iv)(1/2-it)| \asymp 1+N^2,\) thus the factor $1+N^2$ introduced in \eqref{eq:unit-holder} absorbs
the two denominators in ${\mathcal I}_N(x,y)$.

We now fix $N$ and expand the $2k$-th power. The first $k$ copies of
${\mathcal I}_N(x,y)$ and the remaining $k$ copies of its conjugate
introduce variables
\(
 (v_1,t_1),\ldots,(v_{2k},t_{2k}).
\) 
Writing
\[
 \Xi=t_1+\cdots+t_k-t_{k+1}-\cdots-t_{2k},
\]
the $y$-dependent phase is $y^{-i\Xi}$.  For
\(
 y=X^{8/7}e^{\theta+2\pi u},
\)
this becomes \( X^{-8i\Xi/7}e^{-i\theta\Xi}e^{-2\pi iu\Xi}.
\)
The first two factors are independent of $u$ and have modulus one,  summing over the admissible values of $u$ gives, uniformly
in $\theta$, 
\begin{equation}\label{eq:y-geometric-sum}
 \left|\sum_u e^{-2\pi iu\Xi}\right|
 \ll
 \min\left(\log X,\frac1{\|\Xi\|_{\R/\Z}}\right),
\end{equation}
where $\|\cdot\|_{\R/\Z}$ denotes distance to the nearest integer.

By \eqref{eq:P-pointwise-far} and \eqref{eq:near-kernel-mass}, throughout
$|u|\leq d$ we have
\begin{equation}\label{eq:PX-L1-high-moment}
 |P_X(u)|
 \ll\min\left(1,\frac1{|u|\log X}\right),
 \qquad
 \int_{|u|\leq d}|P_X(u)|\,\mathrm{d}u
 \ll\frac{\log\log\log X}{\log X}.
\end{equation}

We first consider
\(
 \|\Xi\|_{\R/\Z}\geq\log^{-2/3}X.
\)
Using \eqref{eq:y-geometric-sum} with the bound $\log^{2/3}X$,
discarding the barrier indicators, and integrating the $v_j$ variables
by \eqref{eq:PX-L1-high-moment}, Harper's argument gives, apart from the
additional factor \eqref{eq:multiplier-2k-loss},
\[
 \begin{aligned}
 &\sum_{y\in{\mathcal G}_X(\theta)}
 |(1+N^2){\mathcal I}_N(x,y)|^{2k}\\
 &\qquad\ll
 (C\log\log\log X)^{8k}\log^{2/3}X
 \left(
 \frac{C\log\log\log X}{\log X}
 \int_N^{N+1}|F_X(1/2+it)|^2\,\mathrm{d}t
 \right)^{2k}\\
 &\qquad \ll \frac{\log X}{\log^{1/3}X}
 \left(
 \frac{C(\log\log\log X)^5}{\log X}
 \int_N^{N+1}|F_X(1/2+it)|^2\,\mathrm{d}t
 \right)^{2k}.
  \end{aligned}
\]
This is the first term in \eqref{eq:weighted-2k}.  In Harper's
Proposition~3 the corresponding power is only
$\log\log\log X$; the additional fourth power here is exactly the loss
coming from the multiplier.

We next turn to \( \|\Xi\|_{\R/\Z}<\log^{-2/3}X.
\)
Since $t_j=N+O(1)$ for every $j$, we have
\[
 \Xi
 =
 \sum_{j\leq k}(t_j-N)-\sum_{j>k}(t_j-N)
 =O(k).
\] 
Hence
\( |\Xi-m|<\log^{-2/3}X
\)
for some integer $|m|\leq Ck$.  Moreover,
\[
 \begin{aligned}
 \bigg|
 \sum_{j\leq k}v_j-\sum_{j>k}v_j-m
 \bigg|
 &\leq
 |\Xi-m|+\sum_{j=1}^{2k}|v_j-t_j|\leq
 \log^{-2/3}X+2kd
 \leq2\log^{-2/3}X
 \end{aligned}
\]
for sufficiently large $X$.

We bound the geometric sum in \eqref{eq:y-geometric-sum} trivially by
$\log X$.  Using
\[
 \begin{aligned}
 &\prod_{j=1}^{2k}
 \mathbf 1_{{\mathcal A}(t_j)}
 \mathbf 1_{{\mathcal A}(v_j)}
 |F_X(\frac12+it_j)F_X(\frac12+iv_j)|\\
 &\qquad\leq
 \frac12\prod_{j=1}^{2k}
 \mathbf 1_{{\mathcal A}(t_j)}
 |F_X(\frac12+it_j)|^2
 +
 \frac12\prod_{j=1}^{2k}
 \mathbf 1_{{\mathcal A}(v_j)}
 |F_X(\frac12+iv_j)|^2
 \end{aligned}
\]
and the symmetry between $v$ and $t$, we need only consider the first
term.  Integrating each $v_j$ by
\eqref{eq:near-kernel-mass}, this part is bounded by
\begin{equation}\label{eq:near-integer-reduction}
 (C\log\log\log X)^{8k}\log X
 \left(\frac{C\log\log\log X}{\log X}\right)^{2k}
 \sum_{|m|\leq Ck}
 \int_{I_m}
 \prod_{j=1}^{2k}
 \mathbf 1_{{\mathcal A}(t_j)}
 |F_X(1/2+it_j)|^2\,\mathrm{d}\mathbf t,
\end{equation}
where $I_m$ denotes the part of the relevant $2k$-dimensional box on
which
\[
 \left|
 \sum_{j\leq k}t_j-\sum_{j>k}t_j-m
 \right|
 \leq2\log^{-2/3}X.
\]

 For each permutation
$\sigma\in S_{2k}$, we restrict to the region
\(
 t_{\sigma(1)}<\cdots<t_{\sigma(2k)}
\)
and further decompose it according to the sizes of the successive
differences.  For $1\leq j<2k$, let $h_j=0$ when
\(
 t_{\sigma(j+1)}-t_{\sigma(j)}
 \leq{e}/{\log X},
\)
and otherwise choose $h_j\geq1$ so that
\( {e^{h_j}}/{\log X}
 <
 t_{\sigma(j+1)}-t_{\sigma(j)}
 \leq
 {e^{h_j+1}}/{\log X}.
\)
As in Harper's proof, we let \( \widetilde h_j
 =
 \min\left(
 h_j,\left\lfloor
 \log\log X-2\log\log\log X
 \right\rfloor
 \right).
\)
Then \( X^{e^{-\widetilde h_j}}= \exp\big(e^{\log\log X-\widetilde h_j}\big)\geq \exp\left((\log\log X)^2\right)\geq1000k^6\) for sufficiently large $X$.  

For $1\leq j<2k$, we write
\[
 F_X(1/2+it_{\sigma(j)})
 =
 F_{X,\widetilde h_j}(1/2+it_{\sigma(j)})
 \frac{F_X(1/2+it_{\sigma(j)})}
 {F_{X,\widetilde h_j}(1/2+it_{\sigma(j)})}.
\]
On ${\mathcal A}(t_{\sigma(j)})$, by the definition of barrier event \eqref{eq:strong-barrier}, we have
\[
 |F_{X,\widetilde h_j}(1/2+it_{\sigma(j)})|
 \leq
 \begin{cases}
 \displaystyle
 \frac{\log X}
 {e^{\widetilde h_j}(\log\log X)^{1000}},
 & \widetilde h_j\leq0.99\log\log X,\\[2mm]
 \displaystyle
 \frac{\log X}{e^{\widetilde h_j}}(\log\log X)^6,
 & \widetilde h_j>0.99\log\log X.
 \end{cases}
\]
Hence
\begin{align}
 &\prod_{l=1}^{2k}\mathbf 1_{{\mathcal A}(t_l)}
 |F_X(\frac12+it_l)|^2\leq
 \prod_{j=1}^{2k-1}
 \left(\frac{\log X}{e^{\widetilde h_j}}\right)^2
 \prod_{\substack{1\leq j<2k\\
 \widetilde h_j\leq0.99\log\log X}}
 (\log\log X)^{-2000}
 \notag\\
 &\qquad\qquad\times \prod_{\substack{1\leq j<2k\\
 \widetilde h_j>0.99\log\log X}}
 (\log\log X)^{12}
 |F_X(\frac12+it_{\sigma(2k)})|^2
 \prod_{j=1}^{2k-1}
 \frac{|F_X(1/2+it_{\sigma(j)})|^2}
 {|F_{X,\widetilde h_j}(1/2+it_{\sigma(j)})|^2}.
 \label{eq:partial-product-bound}
\end{align}
Moreover,
\[
 \frac{F_X(1/2+it_{\sigma(j)})}
 {F_{X,\widetilde h_j}(1/2+it_{\sigma(j)})}
 =
 \prod_{X^{e^{-\widetilde h_j}}<p\leq X}
 \left(1+\frac{f(p)}{p^{1/2+it_{\sigma(j)}}}\right),
\]
so every Euler factor occurring in these quotients corresponds to a
prime $p>1000k^6$.

For a fixed prime $p>1000k^6$, let $J_p$ denote the set of frequencies
whose Euler factors at $p$ occur in the last line of
\eqref{eq:partial-product-bound}.  Thus $|J_p|\leq2k$, and
\[
 |1+f(p)p^{-1/2-it_j}|^2
 =
 1+\frac1p+\frac{2f(p)}{\sqrt p}\cos(t_j\log p).
\]
Expanding the product over $j\in J_p$, averaging over $f(p)$, and then
using $\log(1+u)\leq u$, we obtain $\log \E\prod_{j\in J_p}
 |1+f(p)p^{-1/2-it_j}|^2$ is
\begin{align}
\leq
 \frac{|J_p|}{p}
 +\frac2p
 \sum_{\substack{j<l\\j,l\in J_p}}
 \left(
 \cos((t_l-t_j)\log p)
 +\cos((t_l+t_j)\log p)
 \right)
 +O\left(\frac{k^3}{p^{3/2}}\right).
 \label{eq:rademacher-local-moment}
\end{align}
The primes $p\leq1000k^6$ occur only in the full factor
$F_X(1/2+it_{\sigma(2k)})$, and their contribution to the expectation is \( \prod_{p\leq1000k^6}\left(1+1/p\right)
 \ll\log(1000k^6).
\)

The error terms in \eqref{eq:rademacher-local-moment} are harmless, since \(\sum_{p>1000k^6}{k^3}/{p^{3/2}}\ll1.
\)
The one-point terms are estimated by Mertens' theorem.  For a pair of
Euler product factors whose common prime range is $Q<p\leq X$, Number
Theory Result~2 of \cite[Section~2.2]{Harper2023Large} gives
\[
 \left|
 \sum_{Q<p\leq X}\frac{\cos(u\log p)}p
 \right|
 \ll\frac1{|u|\log Q},
 \qquad
 Q\geq e^{(\log\log X)^2}.
\]

All the frequencies occurring in a fixed $\Delta_N$ have the same sign, since $d=o((\log\log X)^{-2})$.
Hence, for $j<l$,
\[
 |t_{\sigma(l)}+t_{\sigma(j)}|
 \geq
 |t_{\sigma(l)}-t_{\sigma(j)}|.
\]
Thus the terms involving $t_{\sigma(l)}+t_{\sigma(j)}$, which also occur
in Harper's Rademacher case, are bounded by the same estimates as the
terms involving $t_{\sigma(l)}-t_{\sigma(j)}$.  The resulting pair
correlations are therefore controlled by the successive differences
$t_{\sigma(r+1)}-t_{\sigma(r)}$ exactly as in
\cite[Section~3.4]{Harper2023Large}.

From this point, the estimates for the ordered frequency regions and
the summation over the successive gap sizes proceed as in
\cite[Section~3.4]{Harper2023Large}.  After summing over $N$ with the
weights $(1+N^2)^{-1}$, the expectation of the contribution under
consideration, with the deterministic multiplier factor omitted, is
\[
 \ll
 \log X\bigg(
 \frac1{\log^{1/3}X}
 \left(
 Ck^{29}(\log\log X)^{15}\log\log\log X
 \right)^{2k}
 +
 \bigg(
 \frac{Ck^{29}\log\log\log X}
 {(\log\log X)^{991}}
 \bigg)^{2k}
 \bigg).
\]
Applying Markov's inequality at $(\log\log X)^{2k}$ times this
expectation gives an exceptional probability at most
$(\log\log X)^{-2k}$.  Restoring the multiplier factor
\eqref{eq:multiplier-2k-loss} amounts to inserting
$(C\log\log\log X)^4$ inside each $2k$-th power.  Thus the two
quantities inside the powers are bounded by
\(
 Ck^{29}(\log\log X)^{16}(\log\log\log X)^5\le k^{29}(\log\log X)^{17}
\)
and
\(
 {Ck^{29}(\log\log\log X)^5}
 /{(\log\log X)^{990}}\le {k^{29}}/{(\log\log X)^{989}},
\)
respectively, for sufficiently large $X$.
Hence the last two terms are unchanged from Harper's final estimate.
The only visible effect of the deterministic multiplier is the
replacement
\(
 \log\log\log X\)
by \( (\log\log\log X)^5
\)
in the coefficient of the unit-interval mean square term.

Combining the two ranges and then using \eqref{eq:unit-holder} gives
\eqref{eq:weighted-2k}.  The argument is uniform for
$1\leq k\leq\log^{1/4}X$, since, for sufficiently large $X$,
\[
 2kd\leq2\log^{-2/3}X,
 \qquad
 X^{e^{-\widetilde h_j}}\geq e^{(\log\log X)^2}\geq1000k^6.
\]
Finally, the only dependence on the common shift is through
$e^{-i\theta\Xi}$, which has modulus one.  Hence the same exceptional
event works simultaneously for every $0\leq\theta\leq2\pi$.
\end{proof}

\begin{lemma}[Uniform high moment bound]
\label{lem:weighted-high-moment-simple}
For every integer $1\leq k\leq\log^{1/4}X$, outside an event of
probability $O((\log\log X)^{-1})$, we have, simultaneously for every
$0\leq\theta\leq2\pi$,
\begin{equation}\label{eq:weighted-2k-simple-statement}
 \max_{x\in{\mathcal G}_X(\theta)}
 \sum_{y\in{\mathcal G}_X(\theta)}|{\mathcal I}(x,y)|^{2k}
 \ll\log X\bigg(
 \frac{\left(k^{29}(\log\log X)^{17}\right)^{2k}}
 {\log^{1/3}X}
 +
 \bigg(\frac{k^{29}}{(\log\log X)^{989}}\bigg)^{2k}
 \bigg).
\end{equation}
\end{lemma}

\begin{proof}
 By Markov's inequality and a union bound,  we have
\[
 \begin{aligned}
 &\p\left(
 \max_{|N|\leq(\log\log X)^2+1}
 \int_N^{N+1}|F_X(1/2+it)|^2\,\mathrm{d}t
 >
 \log X(\log\log X)^3
 \right)\\
 &\quad\leq
 \frac1{\log X(\log\log X)^3}
 \sum_{|N|\leq(\log\log X)^2+1}
 \int_N^{N+1}
 \E|F_X(1/2+it)|^2\,\mathrm{d}t\\
 &\quad\ll
 (\log\log X)^2(\log\log X)^{-3}
 \ll(\log\log X)^{-1},
 \end{aligned}
\]
where \( \E|F_X(1/2+it)|^2
 =
 \prod_{p\leq X}\left(1+1/p\right)
 \ll\log X
\)
uniformly in $t$.
Thus, outside an event of probability $O((\log\log X)^{-1})$,
\[
 \int_N^{N+1}|F_X(1/2+it)|^2\,\mathrm{d}t
 \leq
 \log X(\log\log X)^3
 \qquad
 (|N|\leq(\log\log X)^2+1).
\]

Substituting this bound into the first term of
\eqref{eq:weighted-2k}, we obtain it is
\[
 \begin{aligned}
\ll
 \frac1{\log^{1/3}X}
 \left(
 C(\log\log\log X)^5(\log\log X)^3
 \right)^{2k}\ll
 \frac1{\log^{1/3}X}
 \left(
 k^{29}(\log\log X)^{17}
 \right)^{2k},
 \end{aligned}
\]
for sufficiently large $X$.  Lemma~\ref{lem:weighted-high-moment}
therefore gives \eqref{eq:weighted-2k-simple-statement}.

The exceptional probability from
Lemma~\ref{lem:weighted-high-moment} is at most
$(\log\log X)^{-2k}\ll(\log\log X)^{-2}$, so together with the
preceding exceptional event the total probability is
$O((\log\log X)^{-1})$.  Both estimates hold simultaneously for every
common shift $\theta$, which completes the proof.
\end{proof}

\begin{proof}[Proof of Proposition~\ref{prop:weighted-covariance}]
We take $k$ to be the largest integer  \(\le
 {\log\log X}/{(5000\log\log\log X)}.
\)
By Lemma~\ref{lem:cov-barrier},
there is an absolute constant $c>0$
such that
\[
 |\cov_X(Y_x,Y_y)|>(\log\log X)^{-0.59}
 \quad\Longrightarrow\quad
 |{\mathcal I}(x,y)|
 \geq c(\log\log X)^{-0.59}.
\]
For fixed $x\in{\mathcal G}_X(\theta)$, let $D_x$ denote the number of
$y\in{\mathcal G}_X(\theta)$ satisfying the latter inequality.  Then
\[
 D_x\,c^{2k}(\log\log X)^{-1.18k}
 \leq
 \sum_{y\in{\mathcal G}_X(\theta)}
 |{\mathcal I}(x,y)|^{2k}.
\]
Lemma~\ref{lem:weighted-high-moment-simple} therefore gives
\[
 \begin{aligned}
 D_x
 &\ll
 c^{-2k}(\log\log X)^{1.18k}\log X
 \bigg(
 \frac{\left(k^{29}(\log\log X)^{17}\right)^{2k}}
 {\log^{1/3}X}
 +
 \bigg(\frac{k^{29}}{(\log\log X)^{989}}\bigg)^{2k}
 \bigg)\\
 &\ll
 \log^{2/3}X
 \left((\log\log X)^{50}\right)^{2k}
 +
 \log X
 \left((\log\log X)^{-950}\right)^{2k}\\
 &\ll
 (\log X)^{0.70}
 \leq(\log X)^{0.8}
 \end{aligned}
\]
for sufficiently large $X$. 

The exceptional events in
Lemmas~\ref{lem:cov-barrier} and
\ref{lem:weighted-high-moment-simple} have probabilities
$O((\log\log X)^{-0.1})$ and
$O((\log\log X)^{-1})$, respectively.  Both conclusions hold
simultaneously for every common shift $\theta$, so their union has
probability $O((\log\log X)^{-0.1})$.  This proves the proposition.
\end{proof}

\section{Conditional Gaussian approximation}\label{sec:fresh-primes}

We now compare the conditional law of the large prime increments
\(Y_x\) in \eqref{eq:fresh-increment} with that of a centred Gaussian
vector having the same conditional covariance matrix. Once the variables
\((f(p))_{p\leq X}\) are fixed, this is a deterministic multivariate
normal-approximation problem. The Gaussian comparison itself is supplied
by Normal Approximation Result~1 of \cite{Harper2023Large}. The only
additional point needed here is to verify that its error terms remain
negligible for the coefficients arising from the extended Rademacher
RMF.

Conditional on \(\mathscr F_X\), let
\((Z_x)_{x\in\mathcal T}\) be a centred Gaussian vector with the same
covariance matrix as \((Y_x)_{x\in\mathcal T}\), and write \(\p_Z\)
for probability with respect to this Gaussian vector.

\begin{lemma}[Conditional Gaussian approximation]
\label{lem:conditional-stein}
Let
\(
\mathcal T_0\subset[X^{8/7},X^{4/3}/4]
\)
be deterministic with
\(\#\mathcal T_0\ll\log X\), and let
\(\varnothing\ne\mathcal T\subseteq\mathcal T_0\) be
\(\mathscr F_X\)-measurable.
Conditional on \(\mathscr F_X\), let
\((Z_x)_{x\in\mathcal T}\) be a centred Gaussian vector with the same
covariance matrix as \((Y_x)_{x\in\mathcal T}\).
Then, for every \(u\in\R\),
\begin{equation}\label{eq:stein-max}
 \p_X\left(\max_{x\in\mathcal T}Y_x\leq u\right)
 \leq
 \p_Z\left(\max_{x\in\mathcal T}Z_x\leq u+1\right)
 +\operatorname{Err}_X(\mathcal T_0),
\end{equation}
where \(\operatorname{Err}_X(\mathcal T_0)\) is
\(\mathscr F_X\)-measurable and satisfies
\begin{equation}\label{eq:stein-error-mean}
 \E\operatorname{Err}_X(\mathcal T_0)
 \ll X^{-1/2}\log^C X
\end{equation}
for some absolute constant \(C>0\).
\end{lemma}

\begin{proof}
Conditional on \(\mathscr F_X\),  we write
\(
 Y_x
 =
 \sum_{X<p\leq X^{4/3}}a_{p,x}f(p),\) where
 \(a_{p,x}\coloneqq {B_X(x/p)}/{\sqrt x}.
\)
The coefficients \(a_{p,x}\) are then deterministic, while the variables
\((f(p))_{p>X}\) remain independent Rademacher random variables.
Normal Approximation Result~1 of \cite{Harper2023Large}, applied
conditionally on \(\mathscr F_X\) with smoothing parameter \(\eta=1\),
therefore gives
\[
 \p_X\big(\max_{x\in\mathcal T}Y_x\leq u\big)
 \leq
 \p_Z\big(\max_{x\in\mathcal T}Z_x\leq u+1\big)
 +O\bigl(E_2(\mathcal T)+E_3(\mathcal T)\bigr),
\]
where
\[
 E_2(\mathcal T)
 \coloneqq
 \sum_{x,y\in\mathcal T}
 \bigg(
 \sum_{X<p\leq X^{4/3}}
 a_{p,x}^2a_{p,y}^2
 \bigg)^{1/2}\quad
\text{and}\quad
 E_3(\mathcal T)
 \coloneqq
 \sum_{X<p\leq X^{4/3}}
 \bigg(
 \sum_{x\in\mathcal T}|a_{p,x}|
 \bigg)^3.
\]
Since every term in these expressions is non-negative, then
\(
 E_j(\mathcal T)\leq E_j(\mathcal T_0)
\) for $j=2,3$.
Thus it suffices to show that
\begin{equation}\label{eq:stein-E2E3-mean}
 \E E_2(\mathcal T_0)+\E E_3(\mathcal T_0)
 \ll X^{-1/2}\log^C X
\end{equation}
for some absolute constant \(C>0\).

We first estimate \(E_2(\mathcal T_0)\). If
\(x\in\mathcal T_0\) and \(p>X\), then
\(
 {4x}/{p}
 \leq X^{1/3}<X.
\)
Hence the moment estimates for \(B_X\) in
Lemma~\ref{lem:complete-moments} apply to \(B_X(x/p)\). For
\(x,y\in\mathcal T_0\), by Cauchy--Schwarz and Lemma~\ref{lem:complete-moments} with $r=4$ we get
\[
 \E\left(a_{p,x}^2a_{p,y}^2\right)
 \ll 
 \frac{1}{xy}\left(\E|B_X({x}/{p})|^4\right)^{1/2} 
\left(\E|B_X({y}/{p})|^4\right)^{1/2}
 \ll
 \frac{(1+x/p)(1+y/p)}{xy}\log^C X.
\]
Moreover, we have \( a_{p,x}a_{p,y}=0\) unless \( p\leq4\min(x,y).
\)
Suppose, without loss of generality, that \(x\leq y\). Then
\[
 \begin{aligned}
 \sum_{X<p\leq X^{4/3}}
 \E\left(a_{p,x}^2a_{p,y}^2\right)
 &\ll
 \frac{\log^C X}{xy}
 \sum_{X<p\leq4x}
 \Big(1+\frac{x}{p}\Big)
 \Big(1+\frac{y}{p}\Big)\\
 &\ll
 \log^C X
 \bigg(
 \frac1y
 +\frac1x\sum_{p>X}\frac1p
 +\sum_{p>X}\frac1{p^2}
 \bigg)\ll X^{-1}\log^C X.
 \end{aligned}
\]
Here we used \(x,y\geq X^{8/7}\).
By Jensen's inequality,
\[
 \begin{aligned}
 \E\bigg(
 \sum_{X<p\leq X^{4/3}}
 a_{p,x}^2a_{p,y}^2
 \bigg)^{1/2}
 &\leq
 \bigg(
 \sum_{X<p\leq X^{4/3}}
 \E(a_{p,x}^2a_{p,y}^2)
 \bigg)^{1/2}\ll X^{-1/2}\log^C X.
 \end{aligned}
\]
Since \(\#\mathcal T_0\ll\log X\), it follows that
\begin{equation}\label{eq:stein-E2-mean}
 \E E_2(\mathcal T_0)
 \ll X^{-1/2}\log^C X.
\end{equation}

We next estimate \(E_3(\mathcal T_0)\). Again by
Lemma~\ref{lem:complete-moments} with  $r=3$,
\[
 \E|a_{p,x}|^3
 \ll
 x^{-3/2}
 \left(1+\frac{x}{p}\right)^{3/2}
 \log^C X,
\]
and \(a_{p,x}=0\) unless \(p\leq4x\). Since
\(
 (1+u)^{3/2}\ll1+u^{3/2},
\)
we obtain
\begin{equation}\label{eq:stein-third-sum}
\begin{aligned}
 &\sum_{X<p\leq X^{4/3}}
 \sum_{x\in\mathcal T_0}\E|a_{p,x}|^3
 \ll \#\mathcal T_0 \sum_{X<p\leq X^{4/3}}\E|a_{p,x}|^3\\
 &\ll
 \log^C X
 \left(
 x^{-3/2}\#\{X<p\leq4x\}
 +
 \sum_{p>X}\frac1{p^{3/2}}
 \right)
 \ll X^{-1/2}\log^C X,
 \end{aligned}
\end{equation}
where  \(\#\mathcal T_0\ll\log X\) has again been absorbed
into the power of \(\log X\).
Now
since
\(
 \left(\sum_{x\in\mathcal T_0}|a_{p,x}|\right)^3
 \leq
 (\#\mathcal T_0)^2
 \sum_{x\in\mathcal T_0}|a_{p,x}|^3,
\)
taking expectations and summing over \(p\), together with
\eqref{eq:stein-third-sum}, we obtain
\begin{equation}\label{eq:stein-E3-mean}
 \E E_3(\mathcal T_0)
 \ll X^{-1/2}\log^C X.
\end{equation}
Combining \eqref{eq:stein-E2-mean} and
\eqref{eq:stein-E3-mean} proves \eqref{eq:stein-E2E3-mean} and hence completes the proof.
\end{proof}

\section{A common logarithmic shift}\label{sec:common-shift}
The estimate in Proposition~\ref{prop:old-prime-energy} controls the
\(X\)-smooth contribution only after averaging over the endpoint, whereas
the high moment argument in Section~\ref{sec:covariances} relies on an
exactly logarithmically spaced grid. We therefore average over common
logarithmic shifts of this grid. Outside a small exceptional event, this
shows that, for all but a small proportion of the shifts, the \(X\)-smooth
contribution is small at most relevant endpoints.

For sufficiently large \(X\), define
\[
 J\coloneqq
 \left\lfloor
 \frac{(4/21)\log X-\log4-2\pi}{2\pi}
 \right\rfloor\asymp\log X
\]
so that $ e^{2\pi}X^{8/7}\leq x_{r,\theta},
\ 4x_{r,\theta}\leq X^{4/3}$, where $ x_{r,\theta}\coloneqq X^{8/7}e^{\theta+2\pi r}$ for \(0\leq\theta<2\pi\) and \(1\leq r\leq J\).
Moreover, we have
\(
 \log x_{r,\theta}-\log x_{s,\theta}=2\pi(r-s),
\) which is
independent of \(\theta\).

\begin{lemma}[Averaging over a common logarithmic shift]
\label{lem:shift-occupancy}
Let \(A>0\). For \(0\leq\theta<2\pi\), let
\(\mathcal R_\theta(A)\) be the set of \(1\leq r\leq J\) such that
\(
 |C_X(x_{r,\theta})|\leq A\) and \( |C_X(4x_{r,\theta})|\leq A.
\)
Then
\begin{equation}\label{eq:mean-bad-occupancy}
 \frac1{2\pi}\int_0^{2\pi}
 \frac{J-\#\mathcal R_\theta(A)}{J}\,\mathrm{d}\theta
 \ll
 \frac1{A^2\log X}
 \int_{X^{8/7}}^{X^{4/3}}
 |C_X(z)|^2\frac{\mathrm{d}z}{z}.
\end{equation}

\end{lemma}

\begin{proof}
Starting with the elementary inequality
\[
 J-\#\mathcal R_\theta(A)
 \leq
 \sum_{r=1}^J
 \left(
 \mathbf 1_{\{|C_X(x_{r,\theta})|>A\}}
 +
 \mathbf 1_{\{|C_X(4x_{r,\theta})|>A\}}
 \right),
\]
we obtain
\[
 \begin{aligned}
 \int_0^{2\pi}
 \bigl(J-\#\mathcal R_\theta(A)\bigr)\,\mathrm{d}\theta
\leq
 \frac1{A^2}
 \int_0^{2\pi}
 \sum_{r=1}^J
 \left(
 |C_X(x_{r,\theta})|^2
 +
 |C_X(4x_{r,\theta})|^2
 \right)\mathrm{d}\theta.
 \end{aligned}
\]
For each of the two families \(x_{r,\theta}\) and \(4x_{r,\theta}\),
the changes of variables
\(
 z=x_{r,\theta}\)
 {and} \(
 z=4x_{r,\theta}
\)
give \(\mathrm{d}z/z=\mathrm{d}\theta\). As \(r\) varies, the
corresponding logarithmic intervals are disjoint apart from their
endpoints and lie in \([X^{8/7},X^{4/3}]\). Hence
\[
 \int_0^{2\pi}
 \sum_{r=1}^J
 \left(
 |C_X(x_{r,\theta})|^2
 +
 |C_X(4x_{r,\theta})|^2
 \right)\mathrm{d}\theta
 \ll
 \int_{X^{8/7}}^{X^{4/3}}
 |C_X(z)|^2\frac{\mathrm{d}z}{z}.
\]
Since \(J\asymp\log X\), this proves
\eqref{eq:mean-bad-occupancy}.
\end{proof}

Taking now
\(A=e^{-1.2W}(\log\log X)^{1/4}\) and using
Proposition~\ref{prop:old-prime-energy}, the
right-hand side of \eqref{eq:mean-bad-occupancy} is
\(
 \ll e^{2.4W}(\log\log X)^{-1/6}.
\) By Markov's inequality, we have
\[
 \#\mathcal R_\theta(A)
 \ge
 \left(1-O\!\left(
 e^{1.2W}(\log\log X)^{-1/12}
 \right)\right)J
\]
for all but
\(O(e^{1.2W}(\log\log X)^{-1/12})\) of the shifts \(\theta\).

\section{From increments to two-sided fluctuations}\label{sec:two-sided}
We now assemble the four ingredients needed to prove Theorem~\ref{thm:one-scale}:
smallness of the \(X\)-smooth contribution on most shifted endpoints,
a lower bound for the conditional variances of the increments, sparsity
of their large conditional covariances, and the conditional Gaussian
approximation. These are supplied, respectively, by
Proposition~\ref{prop:old-prime-energy} together with
Lemma~\ref{lem:shift-occupancy},
Proposition~\ref{prop:B-variance},
Proposition~\ref{prop:weighted-covariance}, and
Lemma~\ref{lem:conditional-stein}.

The exceptional events in these four inputs arise separately. Before
carrying out the Gaussian comparison, we collect them into a single
\(\mathscr F_X\)-measurable good event and restrict to a common set of
shifts on which all four conclusions hold simultaneously. The next lemma records
this reduction.

Throughout this section we write
$A\coloneqq e^{-1.2W}(\log\log X)^{1/4}$.

\begin{lemma}[A simultaneous good event]
\label{lem:aggregate-good-event}
Suppose that \(W_0\leq W\leq(\log\log X)^{1/100}\).
Then there is an \(\mathscr F_X\)-measurable event \(\Omega_X(W)\) satisfying
\begin{equation}\label{eq:old-good-probability}
 \p(\Omega_X(W)^c)
 \ll
 (\log\log X)^{-1/200}
 +e^{-W/10}.
\end{equation}
On \(\Omega_X(W)\), there is a measurable set
\(\Theta_X\subset[0,2\pi)\) of normalised measure
\[
 1-O\left(
 e^{1.2W}(\log\log X)^{-1/12}+X^{-1/8}
 \right)
\]
such that, for every \(\theta\in\Theta_X\), the following hold:
\begin{enumerate}
 \item
 \[
  \#\mathcal R_\theta(A)
  \geq
  \left(
  1-O\left(e^{1.2W}(\log\log X)^{-1/12}\right)
  \right)J.
 \]

 \item For every \(1\leq r\leq J\),
 \begin{equation}\label{eq:increment-variance-input}
  \var_X(Y_{x_{r,\theta}})
  \gg e^{-2.2W}(\log\log X)^{-1/2}.
 \end{equation}

 \item The graph on \(\{1,\ldots,J\}\) in which \(r\ne s\) are joined when
 \begin{equation}\label{eq:bad-covariance-edge}
  \left|
  \cov_X(Y_{x_{r,\theta}},Y_{x_{s,\theta}})
  \right|
  >(\log\log X)^{-0.59}
 \end{equation}
 has maximum degree at most \((\log X)^{0.8}\).

 \item For
 \(\mathcal T_0(\theta)=\{x_{r,\theta}:1\leq r\leq J\}\), we have
 \[
  \operatorname{Err}_X(\mathcal T_0(\theta))\leq X^{-1/8}.
 \]
\end{enumerate}
\end{lemma}

\begin{proof}
By Proposition~\ref{prop:old-prime-energy}, outside an
\(\mathscr F_X\)-measurable event of probability
\(O((\log\log X)^{-1/200})\), we have
\[
 \frac1{\log X}\int_{X^{8/7}}^{X^{4/3}}
 |C_X(z)|^2\frac{\mathrm{d}z}{z}
 \leq(\log\log X)^{1/3}.
\]
Since \(A=e^{-1.2W}(\log\log X)^{1/4}\),
Lemma~\ref{lem:shift-occupancy} then gives
\[
 \begin{aligned}
 \frac1{2\pi}\int_0^{2\pi}
 \frac{J-\#\mathcal R_\theta(A)}{J}\,\mathrm{d}\theta
 &\ll
 A^{-2}(\log\log X)^{1/3}=
 e^{2.4W}(\log\log X)^{-1/6}.
 \end{aligned}
\]
Applying Markov's inequality in \(\theta\) with threshold
\(e^{1.2W}(\log\log X)^{-1/12}\), we see that
\[
 \begin{aligned}
 &\frac1{2\pi}\operatorname{meas}\left\{
 \theta\in[0,2\pi):
 \frac{J-\#\mathcal R_\theta(A)}{J}
 \gg
  e^{1.2W}(\log\log X)^{-1/12}
 \right\}\\
 &\quad \ll
{
 e^{-1.2W}(\log\log X)^{1/12}
 }
 \cdot
 \frac1{2\pi}\int_0^{2\pi}
 \frac{J-\#\mathcal R_\theta(A)}{J}\,\mathrm{d}\theta
 \ll
 e^{1.2W}(\log\log X)^{-1/12}.
 \end{aligned}
\]
This proves the first assertion.

Proposition~\ref{prop:B-variance} gives
\eqref{eq:increment-variance-input}, simultaneously for all relevant
endpoints and hence for every shifted grid, outside an event of probability
\(O(e^{-W/10})\). Likewise,
Proposition~\ref{prop:weighted-covariance} gives the third assertion
simultaneously for every \(\theta\), outside an event of probability
\(O((\log\log X)^{-0.1})\).

It remains to control the Gaussian-approximation error for most shifts.
For each fixed \(\theta\), the set \(\mathcal T_0(\theta)\) is deterministic.
Hence by Lemma~\ref{lem:conditional-stein}, Fubini's theorem, and Markov's
inequality we have
\[
 \begin{aligned}
&\p\left(
 \frac1{2\pi}\int_0^{2\pi}
 \operatorname{Err}_X(\mathcal T_0(\theta))\,\mathrm{d}\theta
 >X^{-1/4}
 \right)
 \ll 
 X^{1/4}\ \E\left[
 \frac1{2\pi}\int_0^{2\pi}
 \operatorname{Err}_X(\mathcal T_0(\theta))\,\mathrm{d}\theta
 \right]\\
 &\qquad=
 X^{1/4}\ \frac1{2\pi}\int_0^{2\pi}
 \E\operatorname{Err}_X(\mathcal T_0(\theta))\,\mathrm{d}\theta
 \ll 
 X^{1/4}\cdot X^{-1/2}\log^C X
 =
 X^{-1/4}\log^C X.
 \end{aligned}
\]
Outside this exceptional event, another application of Markov's inequality,
now in \(\theta\), yields
\[
 \begin{aligned}
 \frac1{2\pi}\operatorname{meas}\left\{
 \theta:
 \operatorname{Err}_X(\mathcal T_0(\theta))>X^{-1/8}
 \right\}
 &\leq
 X^{1/8}
 \frac1{2\pi}\int_0^{2\pi}
 \operatorname{Err}_X(\mathcal T_0(\theta))\,\mathrm{d}\theta\leq X^{-1/8}.
 \end{aligned}
\]
Thus the fourth assertion holds outside a set of shifts of normalised
measure at most \(X^{-1/8}\).

Finally, intersecting the four \(\mathscr F_X\)-measurable good events
gives \(\Omega_X(W)\), and the union bound gives
\eqref{eq:old-good-probability}. Intersecting the sets of shifts arising
from the first and fourth assertions gives \(\Theta_X\), whose normalised
measure is at least
\[
 1-O\left(
 e^{1.2W}(\log\log X)^{-1/12}+X^{-1/8}
 \right).
\]
This completes the proof.
\end{proof}

With Lemma~\ref{lem:aggregate-good-event} in place, we are now in a position to
prove Theorem~\ref{thm:one-scale}.

\begin{proof}[Proof of Theorem~\ref{thm:one-scale}]
Let $W_0$ sufficiently large,  suppose that
\(
 W_0\leq W\leq (1/{30})\log\log\log X.
\)
By Lemma~\ref{lem:aggregate-good-event}, we have
\[
 \p(\Omega_X(W)^c)
 \ll e^{-W/10}+(\log\log X)^{-1/200}\ll e^{-W/10},
\]
and the complement of
\(\Theta_X\) has normalised measure
\[
 \ll e^{1.2W}(\log\log X)^{-1/12}+X^{-1/8}
 \ll(\log\log X)^{-13/300}.
\]
Thus \(|\Theta_X|/(2\pi)=
 1-O\left(
 (\log\log X)^{-13/300}
 \right)\). 

Work on \(\Omega_X(W)\), and fix \(\theta\in\Theta_X\).
By parts (1) and (3) of Lemma~\ref{lem:aggregate-good-event}, and by the greedy algorithm, there exists a measurable
subset \(\mathcal I_\theta\subseteq\mathcal R_\theta(A)\) such that
\begin{equation}\label{eq:independent-set-size}
 \#\mathcal I_\theta
 \geq
 \frac{\left(
  1-O\left(e^{1.2W}(\log\log X)^{-1/12}\right)
  \right)J}{1+(\log X)^{0.8}}
 \geq(\log X)^{0.19}.
\end{equation}

Let \((Z_r)_{r\in\mathcal I_\theta}\) be the centred Gaussian vector
having the same covariance matrix as
\((Y_{x_{r,\theta}})_{r\in\mathcal I_\theta}\).
For distinct \(r,s\in\mathcal I_\theta\), the part (2) and (3) of Lemma~\ref{lem:aggregate-good-event} give
\[
 \begin{aligned}
 \bigg|
 \cov\bigg(
 \frac{Z_r}{\sqrt{\var_X(Y_{x_{r,\theta}})}},
 \frac{Z_s}{\sqrt{\var_X(Y_{x_{s,\theta}})}}
 \bigg)
 \bigg|
 &\ll
 \frac{(\log\log X)^{-0.59}}
 {e^{-2.2W}(\log\log X)^{-1/2}}\ll(\log\log X)^{-1/60}.
 \end{aligned}
\]
Thus, for sufficiently large \(X\), Normal Comparison Result~1 of
\cite{Harper2023Large} applies to these standardised Gaussian variables
with \(\delta=1/100\). Moreover, since
\(\log\#\mathcal I_\theta\geq0.19\log\log X\), after choosing \(W_0\)
sufficiently large we have, uniformly for \(r\in\mathcal I_\theta\),
\[
 \begin{aligned}
 \frac{5A}{\sqrt{\var_X(Y_{x_{r,\theta}})}}
 &\ll
 e^{-0.1W}\sqrt{\log\log X}\leq
 \sqrt{\big(2-\frac1{100}\big)
 \log\#\mathcal I_\theta}.
 \end{aligned}
\]
It follows that
\[
\begin{aligned}
 \p_Z\Big(
 \max_{r\in\mathcal I_\theta}Z_r\leq5A
 \Big)
 &\leq
 \p_Z\bigg(
 \max_{r\in\mathcal I_\theta}
 \frac{Z_r}{\sqrt{\var_X(Y_{x_{r,\theta}})}}
 \leq
 \sqrt{\big(2-\frac1{100}\big)
 \log\#\mathcal I_\theta}
 \bigg)
 =o(1),
 \end{aligned}
\]
by Normal Comparison Result~1 of \cite{Harper2023Large}.

Applying Lemma~\ref{lem:conditional-stein} with
\(\mathcal T=\{x_{r,\theta}:r\in\mathcal I_\theta\}\) and
\(\mathcal T_0=\mathcal T_0(\theta)\), and using
\(\operatorname{Err}_X(\mathcal T_0(\theta))\leq X^{-1/8}\), we obtain
\[
 \begin{aligned}
 &\p_X\Big(
 \max_{r\in\mathcal I_\theta}
 \big(U_X(4x_{r,\theta})-U_X(x_{r,\theta})\big)
 \leq4A
 \Big)\leq
 \p_Z\left(
 \max_{r\in\mathcal I_\theta}Z_r\leq5A
 \right)+X^{-1/8}
 =o(1),
 \end{aligned}
\]
since
\(
 A=e^{-1.2W}(\log\log X)^{1/4}
 \geq(\log\log X)^{21/100},
\)
we have \(4A+1\leq5A\) for sufficiently large \(X\).
Thus
\begin{equation}\label{eq:fresh-increment-max}
 \p_X\Big(
 \max_{r\in\mathcal I_\theta}
 \big(U_X(4x_{r,\theta})-U_X(x_{r,\theta})\big)
 \geq4A
 \Big)
 =1-o(1).
\end{equation}

On the event in \eqref{eq:fresh-increment-max}, the range of the values
\(U_X(x_{r,\theta})\) and \(U_X(4x_{r,\theta})\),
\(r\in\mathcal I_\theta\), is at least \(4A\).
 Conditional on \(\mathscr F_X\), these values form a centrally symmetric
vector, since simultaneously replacing \(f(p)\) by \(-f(p)\) for
\(X<p\leq X^{4/3}\) replaces every $\epsilon_p$ in the definition \eqref{eq:fresh-decomposition} of $U_X$ by its negative. Hence
\[
 \begin{aligned}
 &\p_X\bigg(
 \max_{\substack{r\in\mathcal I_\theta\\
 z\in\{x_{r,\theta},4x_{r,\theta}\}}}
 U_X(z)\geq2A
 \bigg)=
 \p_X\bigg(
 \min_{\substack{r\in\mathcal I_\theta\\
 z\in\{x_{r,\theta},4x_{r,\theta}\}}}
 U_X(z)\leq-2A
 \bigg)
 \geq\frac12-o(1).
 \end{aligned}
\]
Since
\(\mathcal I_\theta\subseteq\mathcal R_\theta(A)\), every endpoint
appearing above satisfies \(|C_X(z)|\leq A\). Recalling
\(S_f(z)/\sqrt z=C_X(z)+U_X(z)\), we obtain, uniformly for
\(\theta\in\Theta_X\),
\[
 \p_X\bigg(
 \max_{\substack{r\in\mathcal I_\theta\\
 z\in\{x_{r,\theta},4x_{r,\theta}\}}}
 \frac{S_f(z)}{\sqrt z}\geq A
 \bigg),\ 
 \p_X\bigg(
 \min_{\substack{r\in\mathcal I_\theta\\
 z\in\{x_{r,\theta},4x_{r,\theta}\}}}
 \frac{S_f(z)}{\sqrt z}\leq-A
 \bigg)
 \geq\frac12-o(1).
\]

All these endpoints lie in \([X^{8/7},X^{4/3}]\). Averaging over the
common shift and using \(|\Theta_X|/(2\pi)=1-o(1)\), we therefore have
\[
 \begin{aligned}
 \p_X\left(
 \max_{X^{8/7}\leq z\leq X^{4/3}}
 \frac{S_f(z)}{\sqrt z}\geq A
 \right)
 &\geq
 \frac1{2\pi}\int_{\Theta_X}
 \p_X\bigg(
 \max_{\substack{r\in\mathcal I_\theta\\
 z\in\{x_{r,\theta},4x_{r,\theta}\}}}
 \frac{S_f(z)}{\sqrt z}\geq A
 \bigg)\mathrm{d}\theta\\
 &\geq
 \frac{|\Theta_X|}{2\pi}
 \left(\frac12-o(1)\right)
 =
 \frac12-o(1)\ge \frac13,
 \end{aligned}
\]
and, in the same way,
\[
 \p_X\left(
 \min_{X^{8/7}\leq z\leq X^{4/3}}
 \frac{S_f(z)}{\sqrt z}\leq-A
 \right)
 \geq\frac13.
\]
Therefore we finish the proof of Theorem~\ref{thm:one-scale}.
\end{proof}

\bibliographystyle{plainnat}
\bibliography{refs}

@article{Wintner1944,
  author  = {Wintner, Aurel},
  title   = {Random factorizations and {R}iemann's hypothesis},
  journal = {Duke Math. J.},
  volume  = {11},
  number  = {2},
  year    = {1944},
  pages   = {267--275},
  doi     = {10.1215/S0012-7094-44-01122-1}
}

@article{kucheriaviy2025,
	title = {Positivity of partial sums of a random multiplicative function and corresponding problems for the Legendre symbol},
	author={Petr Kucheriaviy},
	year = {2025},
    note={Preprint available at arxiv.org/abs/2510.25691}
}

@book{Durrett2019,
  author    = {Durrett, Rick},
  title     = {Probability: Theory and Examples},
  edition   = {5th},
  series    = {Cambridge Series in Statistical and Probabilistic Mathematics},
  volume    = {49},
  publisher = {Cambridge University Press},
  address   = {Cambridge},
  year      = {2019}
}

@article{Soundararajan2000,
  author  = {Soundararajan, Kannan},
  title   = {Nonvanishing of quadratic Dirichlet {$L$}-functions at
             {$s=\frac12$}},
  journal = {Ann. of Math. (2)},
  volume  = {152},
  number  = {2},
  year    = {2000},
  pages   = {447--488},
  doi     = {10.2307/2261390}
}

@misc{HarperSoundararajanXu2026,
  author       = {Harper, Adam J. and Soundararajan, Kannan and Xu, Max Wenqiang},
  title        = {Distribution of random multiplicative functions in short intervals,
                  with proper normalization},
  howpublished = {arXiv:2606.29040},
  year         = {2026},
  url          = {https://arxiv.org/abs/2606.29040}
}

@misc{Verreault2026,
  author       = {Verreault, William},
  title        = {Almost sure upper bound for sums of random multiplicative
                  functions and critical chaos},
  howpublished = {arXiv:2608.21354},
  year         = {2026},
  url          = {https://arxiv.org/abs/2608.21354}
}

@misc{DurkanPearceCrump2026,
  author       = {Durkan, Benjamin and Pearce-Crump, Andrew},
  title        = {A sharp almost sure upper bound for partial sums of random
                  multiplicative functions},
  howpublished = {arXiv:2607.29429},
  year         = {2026},
  url          = {https://arxiv.org/abs/2607.29429}
}

@incollection{Halasz1983,
  author    = {Hal{\'a}sz, G{\'a}bor},
  title     = {On random multiplicative functions},
  booktitle = {Hubert Delange Colloquium (Orsay, 1982)},
  series    = {Publ. Math. Orsay},
  volume    = {83},
  publisher = {Univ. Paris XI},
  address   = {Orsay},
  year      = {1983},
  pages     = {74--96}
}

@article{LauTenenbaumWu2013,
  author  = {Lau, Yuk-Kam and Tenenbaum, G{\'e}rald and Wu, Jie},
  title   = {On mean values of random multiplicative functions},
  journal = {Proc. Amer. Math. Soc.},
  volume  = {141},
  number  = {2},
  year    = {2013},
  pages   = {409--420},
  doi     = {10.1090/S0002-9939-2012-11332-2}
}

@article{Basquin2012,
  author  = {Basquin, Joseph},
  title   = {Sommes friables de fonctions multiplicatives al{\'e}atoires},
  journal = {Acta Arith.},
  volume  = {152},
  number  = {3},
  year    = {2012},
  pages   = {243--266},
  doi     = {10.4064/aa152-3-2}
}

@article{Harper2013Gaussian,
  author  = {Harper, Adam J.},
  title   = {Bounds on the suprema of {G}aussian processes, and omega
             results for the sum of a random multiplicative function},
  journal = {Ann. Appl. Probab.},
  volume  = {23},
  number  = {2},
  year    = {2013},
  pages   = {584--616},
  doi     = {10.1214/12-AAP847}
}

@article{Harper2019High,
  author  = {Harper, Adam J.},
  title   = {Moments of random multiplicative functions, {II}: high moments},
  journal = {Algebra Number Theory},
  volume  = {13},
  number  = {10},
  year    = {2019},
  pages   = {2277--2321},
  doi     = {10.2140/ant.2019.13.2277}
}

@article{Harper2020Low,
  author  = {Harper, Adam J.},
  title   = {Moments of random multiplicative functions, {I}: low moments,
             better than squareroot cancellation, and critical multiplicative
             chaos},
  journal = {Forum Math. Pi},
  volume  = {8},
  year    = {2020},
  pages   = {e1},
  doi     = {10.1017/fmp.2019.7}
}

@article{Harper2023Large,
  author  = {Harper, Adam J.},
  title   = {Almost sure large fluctuations of random multiplicative functions},
  journal = {Int. Math. Res. Not. IMRN},
  volume  = {2023},
  number  = {3},
  year    = {2023},
  pages   = {2095--2138},
  doi     = {10.1093/imrn/rnab299}
}

@misc{Caich2023,
  author       = {Caich, Rachid},
  title        = {Almost sure upper bound for random multiplicative functions},
  howpublished = {arXiv:2304.00943},
  year         = {2023},
  url          = {https://arxiv.org/abs/2304.00943}
}

@article{Mastrostefano2022,
  author  = {Mastrostefano, Daniele},
  title   = {An almost sure upper bound for random multiplicative functions
             on integers with a large prime factor},
  journal = {Electron. J. Probab.},
  volume  = {27},
  year    = {2022},
  pages   = {Paper No. 32, 21 pp.},
  doi     = {10.1214/22-EJP751}
}

@article{Aymone2024,
  author  = {Aymone, Marco},
  title   = {Sign changes of the partial sums of a random multiplicative
             function {II}},
  journal = {C. R. Math.},
  volume  = {362},
  number  = {G8},
  year    = {2024},
  pages   = {895--901},
  doi     = {10.5802/crmath.615}
}

@article{AymoneHeapZhao2023,
  author  = {Aymone, Marco and Heap, Winston and Zhao, Jing},
  title   = {Sign changes of the partial sums of a random multiplicative function},
  journal = {Bull. Lond. Math. Soc.},
  volume  = {55},
  number  = {1},
  year    = {2023},
  pages   = {78--89},
  doi     = {10.1112/blms.12710}
}

@article{GeisHiary2025,
  author  = {Geis, Nick and Hiary, Ghaith},
  title   = {Counting sign changes of partial sums of random multiplicative functions},
  journal = {Q. J. Math.},
  volume  = {76},
  number  = {1},
  year    = {2025},
  pages   = {19--45},
  doi     = {10.1093/qmath/haae061}
}

@misc{Hoystad2026,
  author       = {Sigurd H{\o}ystad},
  title        = {Erd{\H{o}}s Problem \#1144: A Lean 4 proof of
                  almost-sure positive unboundedness},
  year         = {2026},
  month        = sep,
  howpublished = {\url{https://github.com/saasom/Erdos1144}},
  note         = {Released September 6, 2026}
}

@article{Aymone2025Average,
  author  = {Aymone, Marco},
  title   = {Sign changes of the partial sums of a random multiplicative function {III}: Average},
  journal = {Bull. Soc. Math. France},
  volume  = {153},
  number  = {4},
  year    = {2025},
  pages   = {1003--1013},
  doi     = {10.24033/bsmf.2915}
}

@misc{AngeloXu2024,
  author       = {Angelo, Rodrigo and Xu, Max Wenqiang},
  title        = {Oscillations of random multiplicative functions under
                  initial bias},
  howpublished = {arXiv:2411.14447},
  year         = {2026},
  url          = {https://arxiv.org/abs/2411.14447}
}

@misc{Atherfold2025,
  author       = {Atherfold, Christopher},
  title        = {Almost sure bounds for weighted sums of {R}ademacher random
                  multiplicative functions},
  howpublished = {arXiv:2501.11076},
  year         = {2025},
  url          = {https://arxiv.org/abs/2501.11076}
}

@book{MontgomeryVaughan2007,
  author    = {Montgomery, Hugh L. and Vaughan, Robert C.},
  title     = {Multiplicative Number Theory {I}: Classical Theory},
  series    = {Cambridge Studies in Advanced Mathematics},
  volume    = {97},
  publisher = {Cambridge University Press},
  address   = {Cambridge},
  year      = {2007}
}

@article{Huxley1972,
  author  = {Huxley, Martin N.},
  title   = {On the difference between consecutive primes},
  journal = {Invent. Math.},
  volume  = {15},
  year    = {1972},
  pages   = {164--170},
  doi     = {10.1007/BF01418933}
}

@article{Hardy2024Weighted,
  author  = {Hardy, Seth},
  title   = {Almost sure bounds for a weighted {S}teinhaus random multiplicative function},
  journal = {J. Lond. Math. Soc. (2)},
  volume  = {110},
  number  = {3},
  year    = {2024},
  pages   = {e12979},
  doi     = {10.1112/jlms.12979}
}

@unpublished{AtherfoldGerspachHamdan2026,
  author = {Atherfold, Christopher and Gerspach, Maxim and Hamdan, Jad},
  title  = {Low moments of extended {R}ademacher random multiplicative functions},
  year   = {2026},
  note   = {In preparation}
}

@misc{KlurmanLamzouriMunsch2024,
  author       = {Klurman, Oleksiy and Lamzouri, Youness and Munsch, Marc},
  title        = {Sign changes of short character sums and real zeros
                  of {Fekete} polynomials},
  year         = {2024},
  howpublished = {arXiv:2403.02195},
  note         = {Version 2},
  eprint       = {2403.02195},
  archivePrefix = {arXiv},
  primaryClass = {math.NT},
  doi          = {10.48550/arXiv.2403.02195},
  url          = {https://arxiv.org/abs/2403.02195v2}
}

@misc{HardyKlurman2026,
  author = {Seth Hardy and Oleksiy Klurman},
  title  = {Personal communication},
  year   = {2026}
}

@article{Gerspach2022,
  author  = {Gerspach, Maxim},
  title   = {Low Pseudomoments of the Riemann Zeta Function and Its Powers},
  journal = {International Mathematics Research Notices},
  year    = {2022},
  volume  = {2022},
  number  = {1},
  pages   = {625--664},
  doi     = {10.1093/imrn/rnaa159},
}

@inproceedings{Erdos1985,
  author    = {Erd{\H{o}}s, Paul},
  title     = {Some Applications of Probability Methods to Number Theory},
  booktitle = {Proceedings of the 4th Pannonian Symposium on Mathematical Statistics},
  address   = {Bad Tatzmannsdorf, Austria},
  year      = {1985},
  pages     = {1--18},
  publisher = {Reidel},
  location  = {Dordrecht},
  note      = {Mathematical Statistics and Applications, Vol. B; symposium held in 1983}
}

@misc{PaulsProblems1999,
  author = {{Various}},
  title  = {Some of {P}aul's favorite problems},
  year   = {1999},
  note   = {Booklet produced for the conference ``Paul Erd\H{o}s and his
            mathematics'', Budapest, July 1999; see Problem 1.11},
  url    = {https://web.math.pmf.unizg.hr/~vjekovac/EP/Some_of_Pauls_favorite_problems.pdf}
}

@misc{Bloom1144,
  author = {Bloom, Thomas F.},
  title  = {Erd{\H{o}}s Problem \#1144},
  year   = {2026},
  note   = {Accessed 6 September 2026},
  url    = {https://www.erdosproblems.com/1144}
}

\end{document}